\documentclass[12pt,reqno]{amsart}
\usepackage[utf8]{inputenc}
\usepackage{amsmath}
\usepackage{amsfonts}
\usepackage{amssymb}
\numberwithin{equation}{section}
\theoremstyle{plain}
\newtheorem{thm}{Theorem}
\newtheorem{proposition}{Proposition}
\newtheorem{corollary}{Corollary}
\newtheorem{lemma}{Lemma}
\newtheorem{assumption}{Assumption}

\theoremstyle{definition}
\newtheorem{definition}{Definition}
\theoremstyle{remark}
\newtheorem{remark}{Remark}

\newtheorem{example}{Example}

\begin{document}
\title[Nonparametric regression with nuisance components]{A theory of nonparametric regression in the presence 
of complex nuisance components}
\author{Martin Wahl}
\address{Institut f\"{u}r Angewandte Mathematik, Universit\"{a}t Heidelberg,
              Im Neuenheimer Feld 294, 69120 Heidelberg, Germany}
\email{wahl@uni-heidelberg.de}  
\begin{abstract}
In this paper, we consider the nonparametric random regression model $Y=f_1(X_1)+f_2(X_2)+\epsilon$
and address the problem of estimating the function $f_1$. The term $f_2(X_2)$ is regarded as a nuisance term which can be considerably more complex than $f_1(X_1)$. Under minimal assumptions, we prove several nonasymptotic $L^2(\mathbb{P}^X)$-risk bounds for our estimators of $f_1$. Our approach is geometric and based on considerations in Hilbert spaces. It shows that the performance of our estimators is closely related to geometric quantities, 
 such as minimal angles and Hilbert-Schmidt norms. Our results establish new conditions under which the estimators of $f_1$ have up to first order the same sharp upper bound as the corresponding estimators of $f_1$ in the model $Y=f_1(X_1)+\epsilon$. As an example we apply the results to an additive model in which the number of components is very large or in which the nuisance components are considerably less smooth than $f_1$. 
In particular, the results apply to an asymptotic scenario in which the number of components is allowed to increase with the sample size.
\end{abstract}
\maketitle
\section{Introduction} 
In this paper, we consider the nonparametric random regression model
\begin{equation}\label{eq:regrmod}
Y=f_1(X_1)+f_2(X_2)+\epsilon.
\end{equation}
We study the problem of estimating the
function $f_1$, while the function $f_2$ is regarded as a nuisance parameter. We are interested in settings where the second term $f_2(X_2)$ is much more complex than the first term $f_1(X_1)$. A particular model of interest is the additive model
\begin{equation}\label{eq:tmam}
Y=f_1(X_1)+\sum_{j=1}^{q-1}f_{2j}(X_{2j})+\epsilon
\end{equation}
in which the nuisance components $f_{2j}$ are considerably less smooth than $f_1$ or in which the number of components $q$ is very large, for instance in the sense that $q$ is allowed to increase with the sample size $n$.
The estimation problem is similar to the one arising in semiparametric models where the aim is to estimate a finite-dimensional parameter in the presence of a (more complex) infinite-dimensional parameter.

Estimation in nonparametric additive models is a well-studied topic, especially when considering the problem of estimating all components in the case that $q$ is fixed. One of the seminal theoretical papers is by Stone \cite{S}, who showed that each component can be estimated with the rate of convergence corresponding to the situation in which the other components are known. Since then, many estimation procedures have been proposed, many of them consisting of several steps. In the work by Linton \cite{L} and Fan, H\"{a}rdle, and Mammen \cite{FHM}, it is shown that there exist estimators of single components which have the same asymptotic bias and variance as the corresponding oracle estimators for which the other components are known.

Probably the most popular estimation procedures are the backfitting procedures, which are empirical versions of the orthogonal projection onto the subspace of additive functions in a Hilbert space setting (see, e.g., the book by Hastie and Tibshirani \cite{HT} and the references therein). This orthogonal projection was studied, e.g., by Breiman and Friedman \cite{BF} (see also the book by Bickel, Klaassen, Ritov, and Wellner \cite[Appendix A.4]{BKRW}). They showed that, under certain conditions including compactness of certain conditional expectation operators, it can be computed by an iterative procedure using only bivariate conditional expectation operators. Replacing these conditional expectation operators by empirical versions leads to the backfitting procedures. Opsomer and Ruppert \cite{OR} and Opsomer \cite{O} computed the asymptotic bias and variance of estimators based on the backfitting procedure in the case where the conditional expectation operators are estimated using local polynomial regression. Mammen, Linton, and Nielsen \cite{MLN} introduced the smooth backfitting procedure and showed that their estimators of single components achieve the same asymptotic bias and variance as oracle estimators for which the other components are known. Concerning the distribution of the covariates, they make some high-level assumptions which are satisfied under some boundedness conditions on the one- and two-dimensional densities. This is still more than is required in the Hilbert space setting (see \cite{BF}). In the work by Horowitz, Klemel\"{a}, and Mammen \cite{HKM}, a general two-step procedure was proposed in which a preliminary undersmoothed estimator is based on the smooth backfitting procedure of \cite{MLN}. They also showed that there are estimators which are asymptotically efficient (i.e., achieve the asymptotic minimax risk) with the same constant as in the case with only one component. In addition to the assumptions coming from the results in \cite{MLN}, they require a Lipschitz condition for all components.

The problem of estimating $f_1$ in cases in which $f_2(X_2)$ is more complex than $f_1(X_1)$ is also considered in the work by Efromovich \cite{Efr} and Muro and van de Geer \cite{G}. In \cite{Efr}, an estimator of $f_1$ is constructed which is both adaptive to the unknown smoothness and asymptotically efficient with the same constant as in the case with only one component. The assumptions include smoothness and boundedness conditions on the full-dimensional density of $(X_1,X_2)$.  The construction of the estimator is involved and starts with a blockwise-shrinkage oracle estimator. 
In \cite{G}, a penalized least squares estimator is analyzed in cases where the function $f_1$ is smoother than the function $f_2$. Under certain assumptions including smoothness conditions on the design densities, it is shown that for both components, the estimator attains the rate of convergence corresponding to the situation in which the other component is known; i.e., no undersmoothing of the function $f_2$ is needed to estimate the function $f_1$.

The previously discussed literature on additive models focuses on the asymptotic behavior of estimators as the number of observations $n$ goes to infinity in the case that $q$ is fixed. Note that one of our purposes is to generalize several results to the case that $q$ increases with $n$.

Recently, high-dimensional sparse additive models have been studied, e.g., in the work by Meier, van de Geer, and B\"{u}hlmann \cite{MGB}, Huang, Horowitz, and Wei \cite{HHW}, Koltchinskii and Yuan \cite{KY}, Raskutti, Wainwright, and Yu \cite{RWY}, Suzuki and Sugiyama \cite{SuSu}, and Dalalyan, Ingster, and Tsybakov \cite{DIT}. These papers consider the case that the number of covariates $q$ is much larger than the sample size $n$. The focus is on the problem of estimating all components under sparsity constraints. 
In \cite{DIT}, e.g., the authors construct an estimator achieving optimal minimax rates of convergence. These rates of convergence depend on $q$ and also on the smallest degree of smoothness of the $f_{2j}$. Hence, they may only lead to crude bounds for the rates of convergence of estimators of $f_1$. Let us mention that in this paper, we do not consider a sparsity scenario. We are interested in cases in which the number of components $q$ is very large, but smaller than $n$.

In this paper, we consider model \eqref{eq:regrmod} in the case that the functions $f_1$ and $f_2$ belong to closed subspaces $H_1$ and $H_2$ of $\lbrace g_1\in L^2(\mathbb{P}^{X_1}):\mathbb{E}\left[ g_1(X_1)\right]=0 \rbrace$ and $L^2(\mathbb{P}^{X_2})$, respectively. We propose an estimator of $f_1$ which is based on the composition of two least squares criteria. Our main contribution is to derive several nonasymptotic risk bounds which show that the performance of our estimators is closely related to geometric quantities of $H_1$ and $H_2$, such as minimal angles and Hilbert-Schmidt norms. These risk bounds lead to minimal conditions under which the function $f_1$ can be estimated (up to first order) just as well as in
the model $Y=f_1(X_1)+\epsilon$.
Our analysis is based on geometric considerations in Hilbert spaces, and relies on the theory of projections on sumspaces in Hilbert spaces (see, e.g., \cite[Appendix A.4]{BKRW}). Moreover, we apply recent concentration inequalities for structured random matrices (see, e.g., the work by Rauhut \cite{R}) in order to show that several geometric properties in the Hilbert space setting carry over to the finite sample setting with high probability.
As a main example we apply our results to the additive model \eqref{eq:tmam} which corresponds to the case that $H_2$ has an additive structure. Using our results, we establish new conditions on $q$ and on the smoothness of the nuisance components under which our estimator of $f_1$ attains the same (nonasymptotic) optimal rate of convergence as the corresponding least squares estimator in the model $Y=f_1(X_1)+\epsilon$. We also address the question of when the corresponding constants coincide.

The paper is organized as follows. In Section \ref{framew} and \ref{mres}, we present the assumptions on the model and state our main results in Theorems \ref{thm1}-\ref{thm4}. In Section \ref{applifjh}, we apply our results to several models including the additive model. The proofs of our results are given in Sections \ref{prthm12} and \ref{prthm34}. Finally, some complements are given in the Appendix.

\section{The framework}\label{framew}
\subsection{The model} Let $(Y,X_1,X_2)$ be a triple of random variables satisfying \eqref{eq:regrmod},
where $X_1$ and $X_2$ take values in some measurable spaces $(S_1,\mathcal{B}_1)$ and $(S_2,\mathcal{B}_2)$, respectively, $\epsilon$ is a real valued random variable such that $\mathbb{E}\left[  \epsilon\vert X\right]  =0$ and $\mathbb{E}\left[ \epsilon^2\vert X \right]=\sigma^2$, and the unknown regression functions satisfy the following assumption:
\begin{assumption}\label{compass} Suppose that  $f_1\in H_1$, where \[H_1\subseteq\lbrace g_1\in L^2(\mathbb{P}^{X_1}):\mathbb{E}\left[ g_1(X_1)\right]=0 \rbrace\] is a closed subspace, and that $f_2\in H_2$, where $H_2\subseteq L^2(\mathbb{P}^{X_2})$ is a closed subspace. 
\end{assumption} 
Structural assumptions on $f_1$ and $f_2$ (see, e.g., Section \ref{applifjh} where we also consider the additive model) should be incorporated into the model by making assumptions on $H_1$ and $H_2$. From the above, we have that
$X=(X_1,X_2)$ is a random variable taking values in $(S_1\times S_2,\mathcal{B}_1\otimes\mathcal{B}_2)$ (note that in Section \ref{amsc}, we consider the example $S_1=[0,1]$, $S_2=[0,1]^{q-1}$, and $S_1\times S_2=[0,1]^q$, where all spaces are equipped with the Borel $\sigma$-algebra). Moreover, we have that
the spaces $L^2(\mathbb{P}^{X_1})$ and $L^2(\mathbb{P}^{X_2})$ are (in a canonical way) subspaces of $L^2(\mathbb{P}^{X})$, which implies that $H_1$ and $H_2$ are also closed subspaces of $L^2(\mathbb{P}^X)$. Finally, we denote by $f$ the whole regression function given by $f=f_1+f_2$. We assume that we observe $n$ independent copies 
\[(Y^1,X^1),\dots,(Y^n,X^n)\]
of $(Y,X)$, where $X^i=(X_1^i,X_2^i)$, $1\leq i\leq n$. Based on this sample, we consider the problem of estimating the function $f_1$.

\subsection{The main assumption}
Our approach relies strongly on the fact that the space $L^2(\mathbb{P}^X)$ is a Hilbert space with the inner product $\langle g, h\rangle=\mathbb{E}[g(X)h(X)]$ and the corresponding norm $\|g\|=\sqrt{\langle g, g\rangle}$ (see, e.g., \cite[Theorem 5.2.1]{Dudley}). 
In order to state our main assumption, we give the following general definition of a minimal angle in Hilbert spaces (see \cite[Definition 1]{KW} and the references therein).
\begin{definition} Let $\mathcal{H}_1$ and $\mathcal{H}_2$ be two closed subspaces of a Hilbert space $\mathcal{H}$ with inner product $\langle\cdot,\cdot\rangle$ and norm $\|\cdot\|$. The minimal angle between $\mathcal{H}_1$ and $\mathcal{H}_2$ is the number $0\leq \tau_0\leq \pi/2$ whose cosine is given by 
\begin{equation*}
 \rho_0=\rho_0(\mathcal{H}_1,\mathcal{H}_2)=\sup\left\{\frac{\langle h_1,h_2\rangle}{\Vert h_1\Vert\Vert h_2\Vert}\ \bigg|\ 0\neq h_1\in \mathcal{H}_1,0\neq h_2
 \in \mathcal{H}_2\right\}.
\end{equation*}
\end{definition}
\begin{assumption}\label{angle} Suppose that the cosine of the minimal angle between $H_1$ and $H_2$ is strictly less than 1, i.e.,
\begin{equation*}
 \rho_0(H_1,H_2)<1.
\end{equation*}
\end{assumption}
The next lemma states two equivalent
formulations of Assumption \ref{angle}. Since we will also apply it to the finite sample setting in later sections, we again give a general statement. 
\begin{lemma}\label{angleequiv} Let $\mathcal{H}_1$ and $\mathcal{H}_2$ be two closed subspaces of a Hilbert space $\mathcal{H}$ with inner product $\langle\cdot,\cdot\rangle$ and norm $\|\cdot\|$. Let $0\leq\varrho< 1$ be a constant. Then the following assertions are equivalent:
\begin{enumerate}
  \item[(i)] For all $0\neq h_1\in \mathcal{H}_1,0\neq h_2\in\mathcal{H}_2$ we have
  \[\frac{|\langle h_1, h_2\rangle|}{\Vert h_1\Vert\Vert h_2\Vert}\leq \varrho.\]
	\item[(ii)] For all $h_1\in \mathcal{H}_1,h_2\in\mathcal{H}_2$ we have
  \[\Vert h_1+h_2\Vert^2\geq (1-\varrho)(\Vert h_1\Vert^2+\Vert h_2\Vert^2).\]
	\item[(iii)] For all $h_1\in \mathcal{H}_1,h_2\in\mathcal{H}_2$ we have
  \[\Vert h_1+h_2\Vert^2\geq (1-\varrho^2)\Vert h_1\Vert^2.\]
\end{enumerate}
\end{lemma}
A proof of Lemma \ref{angleequiv} is given in Appendix \ref{app2a}.
\subsection{The estimation procedure}\label{es} 
Let $V_1\subseteq H_1$ and $V_2\subseteq H_2$ be $d_1$- and $d_2$-dimensional linear subspaces, respectively, and let $W_1\subseteq V_1$ be a linear subspace. Let $V=V_1+V_2$ and $d=d_1+d_2$. By Assumption \ref{angle}, we have $V_1\cap V_2=\lbrace0\rbrace$, which implies that $d$ is equal to the dimension of $V$ and that each $g\in V$ can be decomposed uniquely as $g=g_1+g_2$ with $g_1\in V_1$ and $g_2\in V_2$. We will make only one assumption on $V$ which relates the $\infty$-norm with the $2$-norm, and which will be needed to apply concentration of measure inequalities (compare to, e.g., \cite[Section 3.1.1]{BM} and \cite[Section 1.1]{B}). 
\begin{assumption}\label{assinfty}
Suppose that there is a real number $\varphi\geq 1$ such that 
\begin{equation}\label{eq:statmodass1}
\Vert g\Vert_\infty\leq\varphi\sqrt{d}\Vert g\Vert
\end{equation} 
for all $g\in V$.
\end{assumption}  
\begin{remark}\label{assinftyrem}
In view of Assumption \ref{angle}, Equation \eqref{eq:statmodass1} is satisfied if there are real numbers $\varphi_j\geq 1$ such that $\Vert g_j\Vert_\infty\leq\varphi_j\sqrt{d_j}\Vert g_j\Vert$ for all $g_j\in V_j$, $j=1,2$. Indeed, applying the Cauchy-Schwarz inequality and Lemma \ref{angleequiv}, we have
\[
\|g_1+g_2\|_\infty\leq \varphi_1\sqrt{d_1}\|g_1\|+\varphi_2\sqrt{d_2}\|g_2\|\leq \frac{\varphi_1\vee\varphi_2}{\sqrt{1-\rho_0}}\sqrt{d_1+d_2}\|g_1+g_2\|.
\]
\end{remark}
The construction of our estimator is based on two least squares criteria. First, let $\hat{f}_V$ be the least squares estimator on the model $V$ which is given (not uniquely) by 
\begin{equation}\label{eq:lscadd}
 \hat{f}_V=\arg\min_{g\in V} \frac{1}{n}\sum_{i=1}^n(Y^i-g(X^i))^2.
\end{equation}  
By the definition of $V$, we have $\hat{f}_V=(\hat{f}_V)_1+(\hat{f}_V)_2$ with $(\hat{f}_V)_1\in V_1$ and $(\hat{f}_V)_2\in V_2$. 
 Next, by applying a second least squares criterion, we define the estimator $\hat{f}_1$ by
\begin{equation}\label{eq:est}
 \hat{f}_1=\arg\min_{g_1\in W_1} \frac{1}{n}\sum_{i=1}^n((\hat{f}_V)_1(X_1^i)-g_1(X_1^i))^2.
\end{equation} 
We will also consider the special case $W_1=V_1$, in which we have $\hat{f}_1=(\hat{f}_V)_1$. This means that the second least squares criterion can be dropped. However, we will see that choosing $V_1$ as a preliminary space of larger dimension leads to a smaller bias (it lowers the dependence on $\rho_0$). Finally, since we want to establish risk bounds, it is convenient to eliminate very large values. Therefore, we define our final estimator $\hat{f}_{1}^{*}$ by
\begin{equation}\label{eq:eststar}
\hat{f}_{1}^{*}=\hat{f}_1\text{ if }\Vert\hat{f}_1\Vert_\infty\leq k_n\text{ and }\hat{f}_{1}^*\equiv 0 \text{ otherwise},
\end{equation}
where $k_n$ is a real number to be chosen later (compare to the work by Baraud \cite[Eq. (3)]{B}). Finally, note that the estimator is not feasible since the distribution of $X$ is not known and therefore the condition $\mathbb{E}\left[ g_1(X_1)\right]=0$ cannot be checked. However, one can replace it by the condition $(1/n)\sum_{i=1}^ng_1(X_1^i)=0$. In Appendix \ref{app2f}, we show how our results carry over to these modified estimators.

In our analysis of $\hat{f}_1^*$, one important step is to carry over the geometric properties valid in the Hilbert space setting to the finite sample setting. For this, the following event
\begin{equation*}
\mathcal{E}_{\delta}=\left\lbrace (1-\delta)\Vert g\Vert^2\leq\Vert g\Vert_n^2\leq (1+\delta)\Vert g\Vert^2\ \text{ for all }g\in V\right\rbrace,
\end{equation*}
$0<\delta<1$, will play the key role. Here, $\|\cdot\|_n$ denotes the empirical norm (see, e.g., Section \ref{prel}). A first observation is that, under Assumptions \ref{compass} and \ref{angle}, the estimator $\hat{f}_1^*$ is unique on the event $\mathcal{E}_\delta$. This can be seen as follows. If $\mathcal{E}_\delta$ holds, then $\Vert\cdot\Vert$ and $\Vert\cdot\Vert_n$ are equivalent norms on $V$, which in turn implies that each $g\in V$ is uniquely determined by $(g(X^1),\dots,g(X^n))^T$. Hence, the solutions of the least squares criteria in \eqref{eq:lscadd} and \eqref{eq:est} are unique (since the solutions are unique when restricted to vectors in $\mathbb{R}^n$ evaluated at the observations). Moreover, by Assumption \ref{angle}, the decomposition $\hat{f}_V=(\hat{f}_V)_1+(\hat{f}_V)_2$ is unique. 

In addition, we also obtain a simple representation of our estimator. Let $\hat{\Pi}_{V}$ be the orthogonal projection from $\mathbb{R}^n$ to the subspace $\lbrace (g(X^1),\dots,g(X^n))^T|g\in V\rbrace$, and let $\hat{\Pi}_{W_1}$ be defined analogously. 
If $\mathcal{E}_\delta$ holds, then we have
\[(\hat{f}_{1}(X_1^1),\dots,\hat{f}_{1}(X_1^n))^T=\hat{\Pi}_{W_1}(\hat{\Pi}_V\mathbf{Y})_1,\]
where $\hat{\Pi}_V\mathbf{Y}=(\hat{\Pi}_V\mathbf{Y})_1+(\hat{\Pi}_V\mathbf{Y})_2$ is the unique decomposition of the least squares estimator on the model $V$, considered as a vector in $\mathbb{R}^n$, with $(\hat{\Pi}_V\mathbf{Y})_j\in\lbrace (g_j(X_j^1),\dots,g_j(X_j^n))^T|g_j\in V_j\rbrace$.

\section{Main results}\label{mres}
\subsection{A first risk bound} In this section, we present a first non\-asymptotic risk bound in the case $W_1=V_1$,
which will be further improved (under additional assumptions) in later sections. We denote by ${\Pi}_V$ (resp. $\Pi_{V_1}$, $\Pi_{V_2}$, and $\Pi_{W_1}$) the orthogonal projection from $L^2(\mathbb{P}^X)$ to the subspace $V$ (resp. $V_1$, $V_2$, and $W_1$).
\begin{thm}\label{thm1} Let Assumption \ref{compass}, \ref{angle}, and \ref{assinfty} be satisfied. Let $0<\delta<1$ be a real number. Let $W_1=V_1$. Then
\begin{align*}
 &\mathbb{E}\left[\Vert f_1-\hat{f}_{1}^*\Vert^2\right]\\& \leq \frac{1+\delta}{(1-\delta)^3}\frac{1}{1-\rho_0^2}
 \left(\left(1+\frac{\varphi^2d}{n}\right)\Vert f-\Pi_{V}f\Vert^2+\frac{\sigma^2 \dim V_1}{n}\right) +R_n
\end{align*} 
with  
\begin{align*}
&R_n=\\&\frac{2(1+\delta)\varphi^2d\Vert f_1\Vert^2(\|f-\Pi_{V_2}f\|^2+\sigma^2)}{(1-\delta)^2(1-\rho_0^2)k_n^2}+2(\Vert f_1\Vert+k_n)^2d
 \exp\left(-\kappa\frac{\delta^2n}{\varphi^2d}\right),
\end{align*}  
where $\kappa$ is the universal constant in Theorem \ref{mfg}.
\end{thm}
Before we discuss the two main terms, let us give conditions under which the remainder term $R_n$ is small. Suppose that for some real number $c>0$, we have
\begin{equation*}
 \varphi^2d\leq \frac{c\delta^2n}{\log n},
\end{equation*}
and let
$k_n^2=\|f_1\|^2n^{\kappa/(2c)}$ (this is a theoretical choice of $k_n$ leading to a simple upper bound for $R_n$, many other choices are possible, too). Then one can show that
\begin{equation*}
R_n\leq\frac{12c(1+\delta)\delta^2}{(1-\delta)^2(1-\rho_0^2)}\left(\|f_1\|^2+\|f_2-\Pi_{V_2}f_2\|^2+\sigma^2\right)n^{-\frac{\kappa}{2c}+1}.
\end{equation*} 
Letting, e.g., $\delta=1/\log n$ and $c=1/\log n$, we obtain the following corollary of Theorem \ref{thm1}.
\begin{corollary}\label{corthm1} Let Assumption \ref{compass}, \ref{angle}, and \ref{assinfty} be satisfied. Suppose that 
\begin{equation}\label{eq:statmodass}
 \varphi^2d\leq \frac{n}{(\log n)^4}.
\end{equation}
Then there is a universal constant $C>0$ such that
\begin{align*}
 &\mathbb{E}\left[\Vert f_1-\hat{f}_{1}^*\Vert^2\right]\\ &\leq \frac{1}{1-\rho_0^2}
 \left(\Vert f_1-\Pi_{V_1}f_1\Vert^2+\frac{\sigma^2 \dim V_1}{n}\right)\left(1+C/\log n\right)\\
&+\frac{C}{1-\rho_0^2}\left( (\log n)\Vert f_2-\Pi_{V_2}f_2\Vert^2+\|f_1\|^2 n^{-\frac{\kappa}{2}\log n+1}\right).
\end{align*}
\end{corollary}
The first two terms on the right hand side are (up to the factor $(1-\rho_0^2)^{-1}$) equal to the bias term and the variance term of the same estimator with $V_2=0$ in the model $Y=f_1(X_1)+\epsilon$. The third term is the approximation error of the function $f_2$ with respect to the space $V_2$. It decreases if $V_2$ is chosen larger. Moreover, the choice of $V_2$ does not effect any of the other terms, the only restriction is given by \eqref{eq:statmodass}. The question arising now is as follows: Is it possible to choose a space $V_2$ subject to the constraint \eqref{eq:statmodass} such that $(1-\rho_0^2)^{-1}(\log n)\Vert f_2-\Pi_{V_2}f_2\Vert^2$ is negligible with respect to the first two terms.
\subsection{A refined risk bound} 
In this section, we improve Theorem \ref{thm1} such that the factor $(1-\rho_0^2)^{-1}$ only appears in remainder terms. Since the refined upper bound for the variance term will also contain a Hilbert-Schmidt norm, we give the following general definition (see, e.g., \cite{Weid}). 
\begin{definition} Let $\mathcal{H}_1$ and $\mathcal{H}_2$ be Hilbert spaces. A bounded linear operator $T:\mathcal{H}_1\rightarrow\mathcal{H}_2$ is called Hilbert-Schmidt if for some orthonormal basis $\{\phi_{1\alpha}\}_{\alpha\in I}$ of $\mathcal{H}_1$,
\begin{equation}\label{eq:hsdef}
\sum_{\alpha\in I}\left\|T\phi_{1\alpha}\right\|^2<\infty.
\end{equation}
This sum is independent of the choice of the orthonormal basis (see  \cite[Satz 3.18]{Weid}).
The square root of this sum is called the Hilbert-Schmidt norm of $T$, denoted by $\|T\|_{HS}$.
\end{definition}
Let $\Pi_{V_2}$ be the orthogonal projection from $L^2(\mathbb{P}^X)$ to $V_2$, and let $\Pi_{V_2}|_{W_1}$ be the restriction of $\Pi_{V_2}$ to $W_1$. Then $\Pi_{V_2}|_{W_1}$ is a Hilbert-Schmidt operator, since $W_1$ is finite-dimensional. We prove:

\begin{thm}\label{thm2} Let Assumption \ref{compass}, \ref{angle}, and \ref{assinfty} be satisfied. Let $0<\delta<1$ be a real number. Then
\begin{align}
 &\mathbb{E}\left[\Vert f_1-\hat{f}_{1}^*\Vert^2\right]\nonumber\\
&\leq\left(\Vert f_1-\Pi_{W_1}f_1\Vert^2+\frac{1}{1-\delta}\frac{\sigma^2\dim W_1}{n}\right) \left(1+\frac{1+\delta}{(1-\delta)^3}\frac{1}{1-\rho_0^2}\frac{2\varphi^2d}{n}\right)\nonumber\\
&  +\frac{1+\delta}{(1-\delta)^2}\frac{6}{1-\rho_0^2}\left(\Vert f_1-\Pi_{V_1}f_1\Vert^2+\Vert f_2-\Pi_{V_2}f_2\Vert^2\right)\nonumber\\
&+\frac{1+\delta}{(1-\delta)^4}\frac{1}{1-\rho_0^2}\frac{\sigma^2\left\|\Pi_{V_2}|_{W_1}\right\|^2_{HS}}{n}+R_n,\label{eq:thm2eq}
\end{align}
where $R_n$ is given in Theorem \ref{thm1}.
\end{thm}
In order to state a corollary of Theorem \ref{thm2} similar to Corollary \ref{corthm1}, we have to discuss the quantity $\left\|\Pi_{V_2}|_{W_1}\right\|^2_{HS}$. If $\{\phi_{1k}\}_{1\leq k\leq \dim W_1}$ is an orthonormal basis of $W_1$, then it can be bounded as follows:
\begin{equation}\label{eq:pog}
\left\|\Pi_{V_2}|_{W_1}\right\|^2_{HS}=\sum_{k=1}^{\dim W_1}\left\|\Pi_{V_2}\phi_{1k}\right\|^2\leq \sum_{k=1}^{\dim W_1}\rho_0^2\|\phi_{1k}\|^2=\rho_0^2\dim W_1,
\end{equation}
where the inequality can be shown as in \eqref{eq:pa1}. Using this bound, we get a variance term which coincides (up to first order) with the one in Theorem \ref{thm1}. However, \eqref{eq:pog} can be considerably improved under certain Hilbert-Schmidt Assumptions. In particular, we will derive upper bounds which are dimension free. The first assumption is as follows:
\begin{assumption}\label{HSce} Suppose that there are measures $\nu_1$ and $\nu_2$ on $\mathcal{B}_1$ and $\mathcal{B}_2$, respectively, such that $X$ has the density $p$ with respect to the product measure $\nu_1\otimes\nu_2$. Let $p_1$ and $p_2$ be the marginal densities of $X_1$ and $X_2$ with respect to the measures $\nu_1$ and $\nu_2$, respectively. Suppose that 
\begin{align*}
 \Vert K\Vert_{HS}^2&=\int_{S_2}\int_{S_1} \left(\frac{p(x_1,x_2)}{p_1(x_1)p_2(x_2)}\right)^2p_1(x_1)p_2(x_2)d\nu_1(x_1)d\nu_2(x_2)\\
&=\int_{S_2}\int_{S_1} \frac{(p(x_1,x_2))^2}{p_1(x_1)p_2(x_2)}d\nu_1(x_1)d\nu_2(x_2)<\infty.
\end{align*}
\end{assumption}
If Assumption \ref{HSce} is satisfied, then we can define the integral operator $K:L^2(\mathbb{P}^{X_1})\rightarrow L^2(\mathbb{P}^{X_2})$ by
\[(K g_1)(x_2)=\int_{S_1} g_1(x_1)\frac{p(x_1,x_2)}{p_1(x_1)p_2(x_2)}p_1(x_1)d\nu_1(x_1)\]
which is the orthogonal projection from $L^2(\mathbb{P}^{X})$ to $L^2(\mathbb{P}^{X_2})$ restricted to $L^2(\mathbb{P}^{X_1})$. Applying \cite[Satz 3.19]{Weid}, we obtain that $K$ is a Hilbert-Schmidt operator with Hilbert-Schmidt norm $\Vert K\Vert_{HS}$. We conclude that
\begin{equation*}
\left\|\Pi_{V_2}|_{W_1}\right\|_{HS}\leq \Vert K\Vert_{HS}.
\end{equation*}
Next, we present a more sophisticated upper bound, by using the spaces $H_1$ and $H_2$ instead of $L^2(\mathbb{P}^{X_1})$ and $L^2(\mathbb{P}^{X_2})$. Let $\Pi_{H_2}$ be the orthogonal projection from $L^2(\mathbb{P}^X)$ to $H_2$, and let $\Pi_{H_2}|_{H_1}$ be the restriction of $\Pi_{H_2}$ to $H_1$. 
\begin{assumption}[Weaker form of Assumption \ref{HSce}]\label{HS} Suppose that $\Pi_{H_2}|_{H_1}$ is a Hilbert-Schmidt operator.
\end{assumption}
If Assumption \ref{HS} is satisfied, then 
\[ 
\left\|\Pi_{V_2}|_{W_1}\right\|_{HS}\leq \left\|\Pi_{H_2}|_{H_1}\right\|_{HS}.
\] 
Letting now $\delta=1/\log n$ and $c=1/\log n$ as in Corollary \ref{corthm1}, we obtain the following corollary of Theorem \ref{thm2}.
\begin{corollary}\label{corthm2} Let Assumption \ref{compass}, \ref{angle}, \ref{assinfty}, and \ref{HSce} be satisfied. Suppose that 
\begin{equation*}
 \varphi^2d\leq \frac{n}{(\log n)^4}.
\end{equation*}
Then there is a universal constant $C>0$ such that
\begin{align*}
 &\mathbb{E}\left[\Vert f_1-\hat{f}_{1}^*\Vert^2\right]\\ &\leq 
 \left(\Vert f_1-\Pi_{W_1}f_1\Vert^2+\frac{\sigma^2 \dim W_1}{n}\right)\left(1+C'/\log n\right)\\
&+C'\left(\Vert f_1-\Pi_{V_1}f_1\Vert^2+\Vert f_2-\Pi_{V_2}f_2\Vert^2+\frac{\sigma^2\left\|K\right\|^2_{HS}}{n}+\frac{\|f_1\|^2}{n^{\frac{\kappa}{2}\log n-1}}\right),
\end{align*}
where $C'=C/(1-\rho_0^2)$. Moreover, if Assumption \ref{HS} holds instead of Assumption \ref{HSce}, then the above inequality holds if $\Vert K\Vert_{HS}^2$ is replaced by $\left\|\Pi_{H_2}|_{H_1}\right\|^2_{HS}$. Finally, if Assumption \ref{HS} and \ref{HSce} are not satisfied, then the above inequality holds if $\Vert K\Vert_{HS}^2$ is replaced by $\rho_0^2\dim W_1$.
\end{corollary}

Now the first two terms in the brackets on the right hand side are equal to the bias term and the variance term of the same estimator with $V_2=0$ in the model $Y=f_1(X_1)+\epsilon$. As in Corollary \ref{corthm1}, we see that the choices of $V_1$ and $V_2$ do not effect any of the other terms, the only restriction is given by \eqref{eq:statmodass}.

Finally, we give an alternative representation of the Hilbert-Schmidt norm $\left\|\Pi_{H_2}|_{H_1}\right\|_{HS}$ using the operator $\Pi_{H_1}\Pi_{H_2}\Pi_{H_1}$ (which we consider as a map from $H_1$ to $H_1$). To simplify the exposition, we suppose that $H_1$ is separable, which implies that each orthonormal basis of $H_1$ is countable (see, e.g., \cite[Chapter II]{RS}). From Assumption \ref{HS}, it follows that $\Pi_{H_1}\Pi_{H_2}\Pi_{H_1}$ is compact (see, e.g., \cite[Chapter 30.8]{La}). Since it is also symmetric and positive, the spectral theorem (see, e.g., \cite[Theorem 3 in Chapter 28]{La}) implies that there is an orthonormal basis for $H_1$ consisting of eigenvectors of $\Pi_{H_1}\Pi_{H_2}\Pi_{H_1}$. These all have non-negative eigenvalues. We arrange the positive eigenvalues of $\Pi_{H_1}\Pi_{H_2}\Pi_{H_1}$ in decreasing order: $\alpha_1\geq \alpha_2\cdots>0$. We now have:

\begin{lemma} Under the above assumptions, we have
\[\rho_0^2=\alpha_1\]
and
\[\left\|\Pi_{H_2}|_{H_1}\right\|^2_{HS}=\operatorname{tr}\left(\Pi_{H_1}\Pi_{H_2}\Pi_{H_1}\right)=\sum_{k\geq 1}\alpha_k.\]
\end{lemma}

\begin{proof}
We only prove the second equality. Let $\{\phi_{1k}\}_{k\geq 1}$ be an orthonormal basis for $H_1$ consisting of eigenvectors of $\Pi_{H_1}\Pi_{H_2}\Pi_{H_1}$. Then 
\begin{align*}
\left\|\Pi_{H_2}|_{H_1}\right\|^2_{HS}&=\sum_{k\geq 1}\|\Pi_{H_2}\phi_{1k}\|^2=\sum_{k\geq 1}\left\langle \Pi_{H_1}\Pi_{H_2}\Pi_{H_1}\phi_{1k},\phi_{1k}\right\rangle^2=\sum_{k\geq 1}\alpha_k.
\end{align*}
\end{proof}
\begin{example}
Consider the case that $X=(X_1,X_2)$ is a bivariate Gaussian random variable such that $\mathbb{E}[X_1]=\mathbb{E}[X_2]=0$, $\mathbb{E}[X_1^2]=\mathbb{E}[X_2^2]=1$, and $\mathbb{E}\left[X_1X_2\right]=\rho$. 

First, suppose that $H_1$ and $H_2$ are the spaces of linear centered functions, i.e., $H_1=\{g_1: g_1(x_1)=a\cdot x_1, a\in\mathbb{R}\}$ and $H_2=\{g_2: g_2(x_2)=a\cdot x_2, a\in\mathbb{R}\}$. Then it is easy to see that 
\[\rho_0=|\rho|\]
and
\[\left\|\Pi_{H_2}|_{H_1}\right\|^2_{HS}=\rho^2.\]

Second, suppose that $H_1=\lbrace g_1\in L^2(\mathbb{P}^{X_1}):\mathbb{E}\left[ g_1(X_1)\right]=0 \rbrace$ and $H_2=L^2(\mathbb{P}^{X_2})$. Then it follows from \cite{Lan} that $\Pi_{H_1}\Pi_{H_2}\Pi_{H_1}$ has eigenvalues $\{\rho^2,\rho^{4},\dots\}$. Hence, the above lemma implies that
\[\rho_0=|\rho|\]
and
\[\left\|\Pi_{H_2}|_{H_1}\right\|^2_{HS}=\sum_{k=1}^\infty\rho^{2k}=\frac{\rho^2}{1-\rho^2},\]
which is an improvement over \eqref{eq:pog} if $\dim W_1$ is large.
\end{example}

\subsection{Regularity conditions on the design densities} In this section, we present two improvements of Theorem \ref{thm2} which are possible under Assumption \ref{HSce} and additional regularity conditions on the design densities. In particular, we show that the dependence of the bias term on the function $f_2$ can decrease considerably. 

By Assumption \ref{HSce} and Fubini's theorem, we have
\begin{equation}\label{eq:condl2}\frac{p(x_1,\cdot)}{p_1(x_1)p_2(\cdot)}\in L^2(\mathbb{P}^{X_2})\end{equation}
for $\mathbb{P}^{X_1}$-almost all $x_1$. Thus we can make the following assumption. Suppose that there is a real number $\psi(V_2)$ and a function $h_1\in L^2(\mathbb{P}^{X_1})$ such that
\begin{equation}\label{eq:condjdens}
\left\Vert(1-\Pi_{V_2}) \frac{p(x_1,\cdot)}{p_1(x_1)p_2(\cdot)}\right\Vert_{L^2(\mathbb{P}^{X_2})}\leq h_1(x_1)\psi(V_2)
\end{equation}
for $\mathbb{P}^{X_1}$-almost all $x_1$. In analogy, we let $\phi(V_2)$ be a real number such that $\Vert f_2-\Pi_{V_2}f_2\Vert\leq\phi(V_2)$. We prove:
\begin{thm}\label{thm3}Let Assumption \ref{compass}, \ref{angle}, \ref{assinfty}, and \ref{HSce} be satisfied. Let $0<\delta<1$ be a real number.
Suppose that \eqref{eq:condjdens} is satisfied. Moreover, suppose that $\Vert g_1\Vert_\infty\leq\varphi\sqrt{d_1}\Vert g_1\Vert$ for all $g_1\in V_1$, where $\varphi$ is the constant from Assumption \ref{assinfty}. Then \eqref{eq:thm2eq} holds when
\[\frac{1+\delta}{(1-\delta)^2}\frac{6}{1-\rho_0^2}\left(\Vert f_1-\Pi_{V_1}f_1\Vert^2+\Vert f_2-\Pi_{V_2}f_2\Vert^2\right)\]
is replaced by
\begin{multline*}
\frac{(1+\delta)^2}{(1-\delta)^4}\frac{12}{1-\rho_0^2}\left(\|f_1-\Pi_{V_1}f_1\|^2+\frac{\Vert h_1\Vert^2(\phi(V_2)\psi(V_2))^2}{1-\rho_{0}^2}\right.\\
 + \left.\frac{1}{n}\frac{\Vert h_1\Vert^2\Vert f_2-\Pi_{V_2}f_2\Vert^2_\infty(\psi(V_2))^2}{1-\rho_{0}^2}+\frac{(\phi(V_2))^2}{1-\rho_{0}^2}\frac{\varphi^2d_1}{n}\right).
\end{multline*}
\end{thm}
Theorem \ref{thm3} shows that the regularity conditions on $p/(p_1p_2)$ and $f_2$ have similar effects, which can be seen from second term. In contrast to Theorems \ref{thm1} and \ref{thm2}, Theorem \ref{thm3} shows that the estimator $\hat{f}_1^*$ can also behave well when $f_2$ is considerably less regular than $f_1$. For instance, if we apply Theorem \ref{thm3} to an asymptotic scenario, then, under suitable conditions on $\psi(V_2)$, the regularity conditions on $f_2$ can be (almost) reduced to $\phi(V_2)\rightarrow 0$ (see, e.g., Corollary \ref{appliabrut}).

For fixed $x_1$, let the function $r(x_1,\cdot)$ be the orthogonal projection of $p(x_1,\cdot)/(p_1(x_1)p_2(\cdot))$ on $H_2$. By \eqref{eq:condl2}, $r(x_1,\cdot)$ is defined for $\mathbb{P}^{X_1}$-almost all $x_1$. Thus we can consider the following weaker version of \eqref{eq:condjdens}.
Suppose that there exists a real number $\psi_\Pi(V_2)$ and a function $h_1\in L^2(\mathbb{P}^{X_1})$ such that 
\begin{equation}\label{eq:condjdenspr}
 \left\| (1-\Pi_{V_2})r(x_1,\cdot)\right\|_{L^2(\mathbb{P}^{X_2})}\leq h_1(x_1)\psi_\Pi(V_2)
\end{equation}
for $\mathbb{P}^{X_1}$-almost all $x_1$. If \eqref{eq:condjdenspr} holds, then we obtain the following theorem. Note that, compared to Theorem \ref{thm3}, the last term is not always negligible.
\begin{thm}\label{thm4} Let Assumption \ref{compass}, \ref{angle}, \ref{assinfty}, and \ref{HSce} be satisfied. Let $0<\delta<1$ be a real number. Suppose that \eqref{eq:condjdenspr} is satisfied. Then Theorem \ref{thm2} also holds when the term 
\[\frac{1+\delta}{(1-\delta)^2}\frac{6}{1-\rho_0^2}\left(\Vert f_1-\Pi_{V_1}f_1\Vert^2+\Vert f_2-\Pi_{V_2}f_2\Vert^2\right)\]
is replaced by
\begin{multline*}
\frac{1+\delta}{(1-\delta)^3}\frac{6}{1-\rho_0^2}\left(\|f_1-\Pi_{V_1}f_1\|^2\left(1+\frac{2\varphi^2d}{n}\right)\right.\\
\left.+\frac{\Vert h_1\Vert^2(\phi(V_2)\psi_\Pi(V_2))^2}{1-\rho_{0}^2}+\frac{(\phi(V_2))^2}{1-\rho_{0}^2}\frac{2\varphi^2 d}{n}\right).
\end{multline*}
\end{thm}
\section{Applications}\label{applifjh}
\subsection{The two-dimensional case}\label{a1dim}
In this section, we want to discuss Theorem \ref{thm1} and \ref{thm2} in the case that $X_1$ and $X_2$ take values in $\mathbb{R}$ and that Assumptions \ref{compass} and \ref{angle} are satisfied with 
\[
H_1=\lbrace g_1\in L^2(\mathbb{P}^{X_1})|\mathbb{E}\left[ g_1(X_1)\right]=0 \rbrace
\] 
and
\[H_2=L^2(\mathbb{P}^{X_{2}}).\]
The main remaining issue is to bound the approximation errors. This is possible if the $f_j$ belong to certain nonparametric classes of functions and if the $V_j$ are chosen appropriately. Here, we shall restrict our attention to (periodic) Sobolev smoothness and spaces of trigonometric polynomials. Note that we will also consider H\"{o}lder smoothness and spaces of piecewise polynomials in Section \ref{aam} and \ref{aamsd}.
Recall that the trigonometric basis is given by $\phi_0(x)= 1$, $\phi_{k}(x)=\sqrt{2}\cos(2\pi kx)$ and $\phi_{-k}(x)=\sqrt{2}\sin(2\pi kx)$, $k\geq 1$, where $x\in[0,1]$. 
\begin{assumption}\label{s1dim}
Suppose that the $X_j$ take values in $[0,1]$ and have densities $p_{X_j}$ with respect to the Lebesgue measure on $[0,1]$, which satisfy $c\leq p_{X_j}\leq 1/c$ for some constant $c>0$. Moreover, suppose that the $f_j$ belong to the Sobolev classes 
\[\tilde{W}_j(\alpha_j,K_j)=\left\{\sum_{k\in \mathbb{Z}}^\infty \theta_k\phi_k(x_j)\ :\ \sum_{k\in \mathbb{Z}}^\infty |k|^{2\alpha_j}\theta_{k}^2\leq K_j^2\right\},\]
where $\alpha_j>0$ and $K_j> 0$ (see, e.g., \cite[Definition 1.12]{T}).
\end{assumption}
For $j=1,2$, let $V_j$ be the intersection of $H_j$ with the linear span of the $\phi_k$ (in the variable $x_j$) such that $|k|\leq m_j$, and let $W_1$ be the intersection of $H_1$ with the linear span of the $\phi_k$ (in the variable $x_1$) such that $|k|\leq m_{W_1}$.  Note that we have $d_1=2m_1$ and $d_2=2m_2+1$. Using the definition of the $\tilde{W}_j(\alpha_j,K_j)$, we have for all $h_j\in\tilde{W}_j(\alpha_j,K_j)\cap H_j$,
\begin{equation*}
\|h_j-\Pi_{V_j}h_j\|^2\leq (1/c)K_j^2(1+m_j)^{-2\alpha_j}
\end{equation*}
and thus
\begin{equation}\label{eq:approx1}
\|h_j-\Pi_{V_j}h_j\|^2\leq C_jK_j^2d_j^{-2\alpha_j},
\end{equation}
where $C_j=2^{2\alpha_j}/c$. The same bound holds if $V_1$ and $d_1$ are replaced by $W_1$ and $\dim W_1$. Moreover, by applying the Cauchy-Schwarz inequality, we have $\|g_j\|_\infty\leq \sqrt{1/c}\sqrt{2m_j+1}\|g_j\|$ for all $g_j\in V_j$, which implies that $\|g_j\|_\infty\sqrt{2/c}\sqrt{d_j}\|g_j\|$ for all $g_j\in V_j$. By Remark \ref{assinftyrem}, this implies that Assumption \ref{assinfty} is satisfied with 
\begin{equation}\label{eq:varphie}
\varphi^2\leq \frac{2}{c(1-\rho_0)}.
\end{equation}
We now choose $d_1$ and $d_2$ as the smallest possible integer satisfying
\begin{equation}\label{eq:choices1}
 d_j\geq \frac{n}{4\varphi^2\log^4 n},
\end{equation} 
and we choose $\dim W_1$ (up to constant) equal to
\[\left(\frac{K_1^2n}{\sigma^2}\right)^{\frac{1}{2\alpha_1+1}}.\]
If $n$ is large enough, then these choices imply that \eqref{eq:statmodass} of Corollary \ref{corthm1} is satisfied. Applying \eqref{eq:approx1} and \eqref{eq:choices1}, we obtain that
\[\|f_2-\Pi_{V_2}f_2\|^2\leq C_2K_2^2\left(\frac{n}{4\varphi^2\log^4 n}\right)^{-2\alpha_2},\]
where the last expression is $o(n^{-(\alpha_1/(2\alpha_1+1)})$ if $\alpha_2>\alpha_1/(2\alpha_1+1)$. Finally, note that $\|f_1\|^2\leq \mathbb{E}[(f_1(X_1)-\theta_0)^2]\leq K_1^2/c$. From Corollary \ref{corthm2}, we now obtain the following asymptotic result when the sample size $n$ tends to infinity:
\begin{corollary}\label{cor11dim} Let Assumption \ref{angle} and \ref{s1dim} be satisfied. Suppose that 
\begin{equation}\label{eq:jkdh}
\alpha_2>\alpha_1/(2\alpha_1+1).
\end{equation}
Then
\begin{equation}\label{eq:aseq1}
\limsup_{n\rightarrow \infty}\sup_{f_j\in \tilde{W}_j(\alpha_j,K_j)}\mathbb{E}\left[n^{\frac{2\alpha_1}{2\alpha_1+1}}\Vert f_1-\hat{f}_1^*\Vert^2\right]\leq
C\sigma^{\frac{4\alpha_1}{2\alpha_1+1}}K_1^{\frac{2}{2\alpha_1+1}},
\end{equation}
where $C$ is a constant depending only on $\alpha_1$, $c$, and $\rho_0$. If, in addition, Assumption \ref{HS} holds, then the dependence of $C$ on $\rho_0$ disappears and we obtain the same constant as for the corresponding estimator with $V_2=0$ in the model $Y=f_1(X_1)+\epsilon$.
\end{corollary}

\begin{remark}
Corollary \ref{cor11dim} says that the estimator $\hat{f}_1^*$ attains the optimal rate of convergence in a minimax sense. Note that one can also take the supremum over random variables $X$ such that $c$, $1/c$, and $1/(1-\rho_0^2)$ are bounded by a fixed constant.
\end{remark}
\begin{remark}\label{relf}
The assumptions of Corollary \ref{cor11dim} are weaker than those needed in \cite{HKM}, where it is also assumed that the joint density of $(X_1,X_2)$ is bounded away from zero and infinity and that the functions $f_1$ and $f_2$ are Lipschitz continuous. Note that \eqref{eq:jkdh} is always satisfied if, e.g., $\alpha_2\geq 1/2$ and that the boundedness conditions imply Assumption \ref{HSce} and also Lemma \ref{angleequiv} (ii). Hence, the boundedness conditions imply Assumption \ref{angle} and \ref{HS}.
\end{remark}
\begin{remark}
In semiparametric models, one often requires a global rate of convergence of order $o(n^{-1/4})$ in order to obtain the rate of convergence $n^{-1/2}$ for the parametric component (see, e.g., the book by van de Geer \cite[Chapter 11]{vdG}). Since $\alpha_1/(2\alpha_1+1)$ goes to $1/2$ as $\alpha_1\rightarrow\infty$, condition \eqref{eq:jkdh} extends this result to the nonparametric case.
\end{remark}
\subsection{The multidimensional case}
Now, we suppose that the $X_1$ and $X_2$ take values in $[0,1]^{q_1}$ and $[0,1]^{q_2}$, respectively, and that Assumptions \ref{compass} and \ref{angle} are again satisfied with
\[H_1=\lbrace g_1\in L^2(\mathbb{P}^{X_1})|\mathbb{E}\left[ g_1(X_1)\right]=0 \rbrace\] and \[H_2=L^2(\mathbb{P}^{X_{2}}).\]
 In this case, we consider the tensor product Fourier basis:
\[\phi_{k}(x_j)=\prod_{l=1}^{q_j}\phi_{k_l}(x_{jl}),\]
where $k\in \mathbb{Z}^{q_j}$ and $x_j\in [0,1]^{q_j}$. We define the following Sobolev class
\[\tilde{W}_j(\alpha_j,K_j)=\left\{\sum_{k\in \mathbb{Z}^{q_j}}^\infty \theta_k\phi_k(x_j)\ :\ \sum_{k\in \mathbb{Z}^{q_j}}^\infty a_{jk}^2\theta_{k}^2\leq K_j^2\right\}\]
with $a_{jk}=\|k\|^{\alpha_j}_\infty$, where $\|k\|_\infty=\max_{l=1,\dots,q_j}|k_l|$, $\alpha_j>0$, and $K_j> 0$ (an alternative choice would be, e.g., $a_{jk}^2=\sum_{l=1}^{q_j}|k_l|^{\alpha_j}$).
\begin{assumption}\label{smdim}
Suppose that the $X_j$ have densities $p_{X_j}$ with respect to the Lebesgue measure on $[0,1]^{q_j}$, which satisfy $c\leq p_{X_j}\leq 1/c$ for some constant $c>0$. Moreover, suppose that the $f_j$ belong to the Sobolev classes $\tilde{W}_j(\alpha_j,K_j)$,
where $\alpha_j>0$ and $K_j> 0$.
\end{assumption}
For $j=1,2$, let $V_j$ be the intersection of $H_j$ with the linear span of the $\phi_k$ (in the variable $x_j$) such that $\|k\|_\infty\leq m_j$, and let $W_1$ be the intersection of $H_1$ with the linear span of the $\phi_k$ (in the variable $x_1$) such that $\|k\|_\infty\leq m_{W_1}$. Note that we have $d_1=(2m_1+1)^{q_1}-1$ and $d_2=(2m_2+1)^{q_2}$. Similarly as above, one can show that
\begin{equation*}
\|h_j-\Pi_{V_j}h_j\|^2\leq C_jK_j^2d_j^{-2\alpha_j/q_j}
\end{equation*}
for all $h_j\in\tilde{W}_j(\alpha_j,K_j)\cap H_j$, where $C_j=2^{2\alpha_j}/c$, and that Assumption \ref{assinfty} holds with a $\varphi$ satisfying again \eqref{eq:varphie}. We now choose $d_1$ and $d_2$ as in \eqref{eq:choices1},
and we choose $\dim W_1$ (up to constant) equal to
\[\left(\frac{K_1^2n}{\sigma^2}\right)^{\frac{q_1}{2\alpha_1+q_1}}.\] 
Similarly as above, we conclude:
\begin{corollary} Let Assumption \ref{angle} and \ref{smdim} be satisfied. Suppose that 
\[
\alpha_2/q_2>\alpha_1/(2\alpha_1+q_1).
\] 
Then
\begin{equation*}
\limsup_{n\rightarrow \infty}\sup_{f_j\in \tilde{W}_j(\alpha_j,K_j)}\mathbb{E}\left[n^{\frac{2\alpha_1}{2\alpha_1+q_1}}\Vert f_1-\hat{f}_1^*\Vert^2\right]\leq
C\sigma^{\frac{4\alpha_1}{2\alpha_1+q_1}}K_1^{\frac{2q_1}{2\alpha_1+q_1}},
\end{equation*}
where $C$ is a constant depending only on $\alpha_1$, $c$, and $\rho_0$. If, in addition, Assumption \ref{HS} holds, then the dependence of $C$ on $\rho_0$ disappears and we obtain the same constant as for the corresponding estimator with $V_2=0$ in the model $Y=f_1(X_1)+\epsilon$.
\end{corollary}
\subsection{The additive model with Sobolev smoothness}\label{amsc}
In this section, we want to discuss Theorem \ref{thm1} and \ref{thm2} in the case that the random variables $X_1$ and $X_2$ take values in $\mathbb{R}$ and $\mathbb{R}^{q-1}$, $q\geq 2$, respectively, and that Assumptions \ref{compass} and \ref{angle} are satisfied with 
\[
H_1=\lbrace g_1\in L^2(\mathbb{P}^{X_1})|\mathbb{E}\left[ g_1(X_1)\right]=0 \rbrace
\] 
and
\[H_2=\sum_{j=1}^{q-1}L^2(\mathbb{P}^{X_{2j}}),\]
where the $X_{2j}$, $j=1,\dots,q-1$, are the components of $X_2$. If we define $H_{2j}=\lbrace g_{2j}\in L^2(\mathbb{P}^{X_{2j}})|\mathbb{E}\left[ g_{2j}(X_{2j})\right]=0 \rbrace$, if $j=1,\dots,q-2$, and $H_{2j}=L^2(\mathbb{P}^{X_{2j}})$, if $j=q-1$, then we can write $H_2=\sum_{j=1}^{q-1}H_{2j}$.
\begin{assumption}\label{hassa}
Suppose that $X_1$ and the $X_{2j}$ take values in $[0,1]$ and have densities $p_{X_1}$ and $p_{X_{2j}}$ with respect to the Lebesgue measure on $[0,1]$, which satisfy $c\leq p_{X_1}\leq 1/c$ and $c\leq p_{X_{2j}}\leq 1/c$ for some constant $c>0$.

Moreover, suppose that $f_1\in\tilde{W}_1(\alpha_1,K_1)$, where $\alpha_1>0$ and $K_1> 0$, and that there is a decomposition $f_2=\sum_{j=1}^{q-1}f_{2j}$ such that $f_{2j}\in\tilde{W}_{2j}(\alpha_2,K_2)\cap H_{2j}$, where $\alpha_2>0$ and $K_2> 0$. 
\end{assumption}
Now, let $V_1$ and $W_1$ be as in Section \ref{a1dim}, and let $V_2=\sum_{j=1}^{q-1}V_{2j}$, where $V_{2j}$ is the intersection of $H_{2j}$ with the linear span of the $\phi_k$ (in the variable $x_{2j}$) such that $|k|\leq m_2$. 
In order to show that Assumption \ref{assinfty} is satisfied, we we need the following additional assumption on $H_2$.
\begin{assumption}\label{ac} There is an
$\epsilon_2<1$ such that for each $h_2\in H_2$, there is a decomposition $h_2=\sum_{j=1}^{q-1}h_{2j}$ with $h_{2j}\in H_{2j}$ such that
\begin{equation}\label{eq:kobcond}
 \Vert h_2\Vert^2\geq(1-\epsilon_2)\sum_{j=1}^{q-1}\|h_{2j}\|^2.
\end{equation}
\end{assumption}
By applying iteratively \cite[Proposition 2.A in Appendix A.4]{BKRW}, one can show that Assumption \ref{ac} is equivalent to the assertion that $\sum_{j\in J}H_{2j}$ is closed for all $J\subseteq \{1,\dots,q-1\}$. In particular, Assumption \ref{ac} implies that $H_2$ is closed meaning that Assumption \ref{compass} is included in Assumption \ref{ac}. We mention that Assumption \ref{ac} can also be related to bounds on certain complementary angles (see \cite[Definition 2 and Proposition 2.D in Appendix A.4]{BKRW}). 
\begin{lemma}\label{varphis} Let Assumption \ref{angle}, \ref{hassa}, and \ref{ac} be satisfied. Then Assumption \ref{assinfty} holds with a constant $\varphi$ satisfying
\[\varphi^2\leq\frac{2}{c(1-\rho_0)(1-\epsilon_2)}\frac{d_1+\sum_{j=1}^{q-1} d_{2j}}{d},\]
where $d_{2j}=\dim V_{2j}$. In particular, if the $V_{2j}$ are linearly independent, then we have
\begin{equation}\label{eq:sdrt}
\varphi^2\leq\frac{2}{c(1-\rho_0)(1-\epsilon_2)}.
\end{equation} 
\end{lemma}
A proof of Lemma \ref{varphis} is given in Appendix \ref{app212}.
Thus \eqref{eq:statmodass} of Corollary \ref{corthm1} is satisfied if 
\[\frac{2\left(d_1+\sum_{j=1}^{q-1}d_{2j}\right)}{c(1-\rho_0)(1-\epsilon_2)}\leq \frac{n}{\log^4 n}.\]
We now choose $\dim W_1$ as in Section \ref{a1dim}, and we choose $d_1$ and the $d_{2j}$ equal to the smallest possible integers satisfying 
\[d_1\geq \frac{c(1-\rho_0)(1-\epsilon_2)n}{8\log^4 n}\]
and
\begin{equation}\label{eq:choicess}
 d_{2j}\geq \frac{c(1-\rho_0)(1-\epsilon_2)n}{8(q-1)\log^4 n}.
\end{equation} 
If the right hand side of \eqref{eq:choicess} is greater than or equal to $2$, then \eqref{eq:statmodass} of Corollary \ref{corthm1} is satisfied. 
In order to bound the approximation error $\Vert f_2-\Pi_{V_2}f_2\Vert$, we introduce $\epsilon_2'$ which is the smallest number such that
\begin{equation}\label{eq:ric3}
\|h_2\|^2\leq(1+\epsilon_2')\sum_{j=1}^{q-1}\left\|h_{2j}\right\|^2
\end{equation}
for all $h_2=\sum_{j=1}^{q-1}h_{2j}$ with $h_{2j}\in H_{2j}$. Note that 
\[1+\epsilon_2'\leq q-1,\]
by the Cauchy-Schwarz inequality.
Using the decomposition of $f_2$ in Assumption \ref{hassa}, the projection theorem, \eqref{eq:ric3}, and finally \eqref{eq:approx1} and \eqref{eq:choicess}, we have
\begin{align*}
\|f_2-\Pi_{V_2}f_2\|^2&\leq \bigg\|\sum_{j=1}^{q-1}(f_{2j}-\Pi_{V_{2j}}f_{2j})\bigg\|^2\\
&\leq (1+\epsilon_2')\sum_{j=1}^{q-1}\|f_{2j}-\Pi_{V_{2j}}f_{2j}\|^2\\
&\leq (1+\epsilon_2')C_2K_2^2(q-1)\left(\frac{c(1-\rho_0)(1-\epsilon_2)n}{8(q-1)\log^4 n}\right)^{-2\alpha_2}.
\end{align*}
From Corollary \ref{corthm2}, we now obtain:
\begin{thm}\label{addmcors} Let Assumption \ref{angle}, \ref{hassa}, and \ref{ac} be satisfied. Suppose that the right hand side of \eqref{eq:choicess} is greater than or equal to $2$ and that $\dim W_1\leq d_1$.
Then
\begin{multline}\label{eq:kfori}
\sup_{f_1\in \tilde{W}_1(\alpha_1,K_1)}\sup_{f_{2j}\in \tilde{W}_{2j}(\alpha_2,K_2)}\mathbb{E}\left[\Vert f_1-\hat{f}_1^*\Vert^2\right]\leq
C\sigma^{\frac{4\alpha_1}{2\alpha_1+1}}K_1^{\frac{2}{2\alpha_1+1}}n^{-\frac{2\alpha_1}{2\alpha_1+1}}\\
+C'(1+\epsilon_2')q^{2\alpha_2+1}(\log n)^{8\alpha_2}((1-\epsilon_2)n)^{-2\alpha_2},
\end{multline}
where $C$ is a constant depending only on $\alpha_1$, $c$ and $\rho_0$, and $C'$ is a constant depending only on $\alpha_2$, $c$, $\rho_0$, $K_2$, and $K_1$.
\end{thm} 
\begin{remark}
If the last expression in \eqref{eq:kfori} is of smaller order, then Theorem \ref{addmcors} says that the estimator $\hat{f}_1^*$ attains the (nonasymptotic) optimal rate of convergence in a minimax sense. The last expression in \eqref{eq:kfori} is of smaller order if, e.g.,
\begin{equation}\label{eq:dkit}
q^{2\alpha_2+1}(\log n)^{8\alpha_2+1}n^{-2\alpha_2}\leq C''n^{-\frac{2\alpha_1}{2\alpha_1+1}},
\end{equation}
where $C''$ depends only on $\alpha_2$, $c$, $\rho_0$, $K_2$, $K_1$, $\epsilon_2$, and $\epsilon_2'$. This result can be applied to an asymptotic scenario in which $q$ and $n$ tend to infinity such that \eqref{eq:dkit} is satisfied.
\end{remark}
Next, we apply Theorem \ref{addmcors} in the case that the sample
size $n$ tends to infinity, and $q$ is a fixed constant.
\begin{corollary}\label{coramod} Let Assumption \ref{angle}, \ref{hassa}, and \ref{ac} be satisfied. Suppose that 
\[
\alpha_2>\alpha_1/(2\alpha_1+1).
\] 
Then
\begin{equation*}
\limsup_{n\rightarrow \infty}\sup_{f_1\in \tilde{W}_1(\alpha_1,K_1)}\sup_{f_{2j}\in \tilde{W}_{2j}(\alpha_2,K_2)}\mathbb{E}\left[n^{\frac{2\alpha_1}{2\alpha_1+1}}\Vert f_1-\hat{f}_1^*\Vert^2\right]\leq
C\sigma^{\frac{4\alpha_1}{2\alpha_1+1}}K_1^{\frac{2}{2\alpha_1+1}},
\end{equation*}
where $C$ is a constant depending only on $\alpha_1$, $c$, and $\rho_0$. If, in addition, Assumption \ref{HS} holds, then the dependence of $C$ on $\rho_0$ disappears and we obtain the same constant as for the corresponding estimator with $V_2=0$ in the model $Y=f_1(X_1)+\epsilon$.
\end{corollary}
\begin{remark}
Again, the assumptions of Corollary \ref{coramod} are weaker than those needed in \cite{HKM} (compare to Remark \ref{relf}). For instance, the condition that the one- and two-dimensional marginal densities of $X_{2j}$ and $(X_{2j},X_{2j'})$ are bounded away from zero and infinity implies that Assumption \ref{ac} is satisfied (see, e.g., \cite[Proposition 2.C in Appendix A.4]{BKRW} for the case $q=3$ and \cite[Lemma 1]{MLN} for the general case).
\end{remark}

\subsection{The additive model with H\"{o}lder smoothness}\label{aam} We continue the discussion of the additive model in Section \ref{amsc}, and consider briefly the case of H\"{o}lder smoothness and spaces of piecewise polynomials. 
\begin{assumption}\label{hassh}
Suppose that $X_1$ and the $X_{2j}$ take values in $[0,1]$ and have densities $p_{X_1}$ and $p_{X_{2j}}$ with respect to the Lebesgue measure on $[0,1]$, which satisfy $p_{X_1}\geq c$ and $p_{X_{2j}}\geq c$ for some constant $c>0$.

Moreover, suppose that the function $f_1$ belongs to the H\"{o}lder class $\mathcal{H}_1(\alpha_1,K_1)$ on $[0,1]$, where $\alpha_1>0$ and $K_1\geq 0$ (see, e.g., \cite[Definition 1.2]{T}), and that there is a decomposition $f_2=\sum_{j=1}^{q-1}f_{2j}$ such that $f_{2j}\in\mathcal{H}_{2j}(\alpha_2,K_2)\cap H_{2j}$, where $\alpha_2>0$ and $K_2\geq 0$. 
\end{assumption}
Let $V_1$ be the intersection of $H_1$ with the space of regular piecewise polynomials (in the variable $x_1$) with integer-valued parameters $r_1=\left\lfloor \alpha_1\right\rfloor$ and $m_1$, where $r_1$ is the maximal degree of the polynomials and $\{0<1/m_1<2/m_1<\dots<1\}$ generates the partition of $[0,1]$ into $m_1$ intervals (see, e.g., \cite{BM2}), and let $W_1$ be the intersection of $H_1$ with the space of regular piecewise polynomials (in the variable $x_1$) with integer-valued parameters $r_1=\left\lfloor \alpha_1\right\rfloor$ and $m_{W_1}$ (in order that $W_1\subseteq V_1$ we need that $m_1$ is a multiple of $m_{W_1}$).
Moreover, let $V_2=\sum_{j=1}^{q-1}V_{2j}$, where $V_{2j}$ is the intersection of $H_{2j}$ with the space of regular piecewise polynomials (in the variable $x_{2j}$) with integer-valued parameters $r_2=\left\lfloor \alpha_2\right\rfloor$ and $m_2$. Note that alternatively, one could also consider spaces of splines with the same parameters (see, e.g., \cite[Chapter VII, VIII]{dB}). We have $d_1=(r_1+1)m_1-1$, $d_{2j}=(r_2+1)m_2-1$, if $j=1,\dots,q-1$, and $d_{2j}=(r_2+1)m_2$, if $j=q-1$. Using Taylor's theorem, one can show that there are constants $C_1$ and $C_2$ depending only on $\alpha_1$ and $\alpha_2$, respectively, such that
\begin{equation}\label{eq:apprpr1}
\inf _{g_1\in V_1}\Vert h_1-g_1\Vert_\infty^2\leq C_1K_1^2d_1^{-2\alpha_1}
\end{equation}
for all $h_1\in \mathcal{H}(\alpha_1,K_1)\cap H_1$ and 
\begin{equation}\label{eq:apprpr2}
\inf _{g_{2j}\in V_{2j}}\Vert h_{2j}-g_{2j}\Vert_\infty^2\leq C_2K_2^2d_{2j}^{-2\alpha_2}
\end{equation}
for all $h_{2j}\in \mathcal{H}(\alpha_2,K_2)\cap H_{2j}$. Note that similar bounds also hold for spline spaces (see, e.g., \cite[Chapter XII]{dB}). Concerning Assumption \ref{assinfty}, we have a similar result as in the previous section:
\begin{lemma}\label{varphi} Let Assumptions \ref{angle} and \ref{ac}, and \ref{hassh} be satisfied. Let $r=r_1\vee r_2$. Then we have
\[\varphi^2\leq\frac{2(r+1)}{c(1-\rho_0)(1-\epsilon_2)}\frac{d_1+\sum_{j=1}^{q-1} d_{2j}}{d}.\]
In particular, if the $V_{2j}$ are linearly independent, then we have
\[\varphi^2\leq\frac{2(r+1)}{c(1-\rho_0)(1-\epsilon_2)}.\]
\end{lemma}
A proof of Lemma \ref{varphi} is given in Appendix \ref{app212}. Choosing $d_1$, $\dim W_1$, and $d_{2j}$, $j=1,\dots,q-1$, similarly as in the previous section, we obtain the following analogue of Theorem \ref{addmcors}:
\begin{thm}\label{addmcor} Let Assumption \ref{angle}, \ref{ac}, and \ref{hassh} be satisfied. Suppose that the right hand side of \eqref{eq:choicess} is greater than or equal to $r+1$ and that $\dim W_1\leq c(1-\rho_0)(1-\epsilon_2)n/(4(r+1)\log^4 n)$.
Then
\begin{align*}
\mathbb{E}\left[\Vert f_1-\hat{f}_1^*\Vert^2\right]&\leq
C\sigma^{\frac{4\alpha_1}{2\alpha_1+1}}K_1^{\frac{2}{2\alpha_1+1}}n^{-\frac{2\alpha_1}{2\alpha_1+1}}\\
&+C'(1+\epsilon_2')\varphi^{4\alpha_2}q^{2\alpha_2+1}(\log n)^{8\alpha_2+1}((1-\epsilon_2)n)^{-2\alpha_2}
\end{align*}
where $C$ is a constant depending only on $\alpha_1$ and $\rho_0$, and $C'$ is a constant depending only on $\alpha_2$, $c$, $r$, $\rho_0$, $K_2$, and $\|f_1\|$.
\end{thm}

\subsection{The additive model with smooth design densities}\label{aamsd}  We continue the example in Section \ref{aam} and discuss Condition \eqref{eq:condjdens} and \eqref{eq:condjdenspr} in Theorem \ref{thm3} and \ref{thm4}, respectively. First, we apply Theorem \ref{thm3} in the simple case $q=2$. We suppose that Assumption \ref{HSce} holds and that for each fixed $x_1$, 
\begin{equation}\label{eq:p2joint}
 \frac{p_X(x_1,x_2)}{p_{X_1}(x_1)p_{X_2}(x_2)}\in\mathcal{H}(\beta,h_1(x_1))
\end{equation}
with $h_{1}\in L^2(\mathbb{P}^{X_1})$. Let $V_1$ and $W_1$ as in the previous section, and let $V_2$ be the space of regular piecewise polynomials in the variable $x_{2}$ with parameters $r_2=\lfloor \alpha_2\rfloor\vee \lfloor\beta\rfloor$ and $d_2$ as in \eqref{eq:choices1}. Then \eqref{eq:apprpr2} holds with a constant $C_2$ depending only on $\alpha_2$ and $r_2$. 
Moreover, \eqref{eq:condjdens} is satisfied with $h_1$ from \eqref{eq:p2joint} and $\psi(V_2)=\sqrt{C_3}d_2^{-\beta}$, where $C_3$ is a constant depending only on $\beta$ and $r_2$. Thus
\begin{equation*}
 \Vert h_1\Vert\psi(V_2)\phi(V_2)\leq \sqrt{C_{2}C_3}K_2\Vert h_1\Vert\left(\frac{n}{4\varphi^2\log^4 n}\right)^{-\alpha_2-\beta}
\end{equation*}
From Theorem \ref{thm3}, we obtain:
\begin{corollary}\label{appliabrut} Let $q=2$. Let Assumption \ref{angle}, \ref{hassh}, and \ref{HSce} be satisfied. Suppose that \eqref{eq:p2joint} holds, and that 
\begin{equation}\label{eq:dlskh}
\alpha_2+\beta>\alpha_1/(2\alpha_1+1).
\end{equation}
Then 
\begin{equation}\label{eq:aseq}
\limsup_{n\rightarrow \infty}\mathbb{E}\left[n^{\frac{2\alpha_1}{2\alpha_1+1}}\Vert f_1-\hat{f}_1^*\Vert^2\right]\leq
C\sigma^{\frac{4\alpha_1}{2\alpha_1+1}}K_1^{\frac{2}{2\alpha_1+1}},
\end{equation}
where $C$ is a constant depending only on $\alpha_1$.
\end{corollary}
Finally, we apply Theorem \ref{thm2}. In the particular case that the $X_{2j}$, $j=1,\dots,q-1$, are independent, one can show that (see \eqref{eq:condexf})
\begin{equation*}
 r(x_1,x_2)=\sum_{j=1}^{q-1}\frac{p_{X_1,X_{2j}}(x_1,x_{2j})}{p_{X_1}(x_1)p_{X_{2j}}(x_{2j})}-(q-2),
\end{equation*} 
where $p_{X_1,X_{2j}}$ denotes the joint density of $(X_1,X_{2j})$. In this particular case, condition \eqref{eq:condjdenspr} is thus much weaker than condition \eqref{eq:condjdens}: we only need smoothness conditions on several kernels which involve only one- and two-dimensional design densities. In the general case, we have a similar result.
Suppose that for each fixed $x_1$
\begin{equation}\label{eq:jointdens}
 \frac{p_{X_1,X_{2k}}(x_1,x_{2k})}{p_{X_1}(x_1)p_{X_{2k}}(x_{2k})}\in\mathcal{H}(\beta,h_{1k}(x_1))
\end{equation}
with $h_{1k}\in L^2(\mathbb{P}^{X_1})$, for all $k=1,\dots,q-1$. Moreover, suppose that for each fixed $x_{2j}$
\begin{equation}\label{eq:jointdens2}
 \frac{p_{X_{2j},X_{2k}}(x_{2j},x_{2k})}{p_{X_{2j}}(x_{2j})p_{X_{2k}}(x_{2k})}\in\mathcal{H}(\beta,h'_{jk}(x_{2j}))
\end{equation}
with $h'_{jk}\in L^2(\mathbb{P}^{X_{2j}})$, for all $j,k=1,\dots,q-1$. Then we have:
\begin{corollary}\label{corlemmab} Let $q>2$. Let Assumption \ref{angle}, \ref{ac}, \ref{hassh}, and \ref{HSce} be satisfied. Suppose that \eqref{eq:jointdens} and \eqref{eq:jointdens2} are satisfied. Moreover, suppose that $\alpha_2+\beta/(2\alpha_1+1)>\alpha_1/(2\alpha_1+1)$. Then \eqref{eq:aseq} holds, where $C$ is a constant depending only on $\alpha_1$.
\end{corollary}
A proof of Corollary \ref{corlemmab} is given in Appendix \ref{app2b}. Note that in Corollary \ref{corlemmab}, the smoothness condition is stronger than the smoothness condition given in \eqref{eq:dlskh}. 

\section{Proof of Theorem \ref{thm1} and \ref{thm2}}\label{prthm12}
\subsection{The finite sample geometry}\label{prel} 
In this section, we present an empirical version of Assumption \ref{angle}, which holds on the event
\begin{equation*}\mathcal{E}_{\delta}=\left\lbrace (1-\delta)\Vert g\Vert^2\leq\Vert g\Vert_n^2\leq (1+\delta)\Vert g\Vert^2\ \text{ for all }g\in V\right\rbrace,
\end{equation*} 
where $0<\delta<1$ is the constant from Theorem \ref{thm1}.
Moreover, using concentration of measure inequalities for structured random matrices, we lower bound the probability that the event $\mathcal{E}_\delta$ occurs. 

In order to state our first result, we introduce the empirical inner product $\langle \cdot,\cdot\rangle_n$ and the corresponding empirical norm $\|\cdot\|_n$ which are given by 
\[\langle g,h\rangle_n=\frac{1}{n}\sum_{i=1}^n g(X^i)h(X^i)\]
and $\|g\|_n^2=\langle g,g \rangle_n$ for $g,h\in L^2(\mathbb{P}^X)$.
\begin{proposition}\label{empangle} Let Assumptions \ref{compass} and \ref{angle} hold. If $\mathcal{E}_\delta$ holds, then we have
\begin{align}
 \Vert g_1+g_2\Vert^2_n&\geq \frac{(1-\delta)}{(1+\delta)}(1-\rho_0)(\Vert g_1\Vert^2_n+\Vert g_2\Vert^2_n)\label{eq:empang}\ \ \text{ and}\\
 \Vert g_1+g_2\Vert^2_n&\geq \frac{(1-\delta)}{(1+\delta)}(1-\rho_0^2)\Vert g_1\Vert^2_n\label{eq:empangs}
\end{align}
for all $g_1\in V_1$, $g_2\in V_2$, and also
\begin{equation}\label{eq:backfcond}
 \frac{|\langle g_1,g_2\rangle_n|}{\Vert g_1\Vert_n\Vert g_2\Vert_n}\leq 1-\frac{(1-\delta)}{(1+\delta)}(1-\rho_0)
\end{equation} 
for all $0\neq g_1\in V_1$, $0\neq g_2\in V_2$.
\end{proposition}
\begin{proof}
Let $g_1\in V_1$ and $g_2\in V_2$. By the definition of $\mathcal{E}_\delta$ and Lemma \ref{angleequiv} combined with Assumption \ref{angle}, we have
\begin{align*}
 \Vert g_1+g_2\Vert^2_n \geq (1-\delta)\Vert g_1+g_2\Vert^2
   						&\geq (1-\delta)(1-\rho_0)(\Vert g_1\Vert^2+\Vert g_2\Vert^2)\\
   						&\geq \frac{(1-\delta)}{(1+\delta)}(1-\rho_0)(\Vert g_1\Vert^2_n+\Vert g_2\Vert^2_n)
\end{align*}
and similarly
\begin{align*}
 \Vert g_1+g_2\Vert^2_n \geq (1-\delta)\Vert g_1+g_2\Vert^2
   						&\geq (1-\delta)(1-\rho_0^2)\Vert g_1\Vert^2\\
   						&\geq \frac{(1-\delta)}{(1+\delta)}(1-\rho_0^2)\Vert g_1\Vert^2_n.
\end{align*}
This gives \eqref{eq:empang} and \eqref{eq:empangs}. \eqref{eq:backfcond} follows from \eqref{eq:empang} and Lemma \ref{angleequiv}.
This completes the proof.
\end{proof}
The following result follows from Rauhut \cite[Theorem 7.3]{R} (see also \cite[Theorem 3.1]{RV}). It can also be obtained from a combination of Talagrand's inequality and Rudelson's lemma (see \cite[Theorem 1]{Rud}).
\begin{thm}\label{mfg} Let Assumption \ref{assinfty} hold. Then we have
\[\mathbb{P}\left(\mathcal{E}_\delta\right)\geq 1-2^{3/4}d\exp\left( -\kappa\frac{n\delta^2}{\varphi^2d}\right),\]
where $\kappa$ is a universal constant. 
\end{thm}
\begin{proof}
Let $b_1,\dots,b_d$ be an orthonormal basis of $V$. By Assumption \ref{assinfty} and \cite[Lemma 1]{BM}, we have
\begin{equation}\label{eq:leq}
\Big\|\sum_{j=1}^{d}b_j^2\Big\|_{\infty}\leq \varphi^2d.
\end{equation}
Now let
\[B_n=(\langle b_j,b_k\rangle_n)_{1\leq j,k\leq d}.\]
Then \cite[Theorem 7.3]{R} (in the case $s=N=d$) yields for $0<\delta<1$, 
\begin{equation}\label{eq:rerav}
\mathbb{P}(\Vert B_n-I\Vert_{\operatorname{op}} \leq \delta)\geq 1- 2^{3/4}d\exp\left( -\kappa\frac{n\delta^2}{\varphi^2d}\right),
\end{equation}
where $\kappa>0$ is a universal constant. Here, $\Vert\cdot\Vert_{\operatorname{op}}$ denotes the operator norm (see Lemma \ref{trlem} (iii) for the definition).
 Note that we can apply \cite[Theorem 7.3]{R} since in the proof the condition \cite[(4.2)]{R} is only used in the form \cite[(7.5)]{R}, which is satisfied by \eqref{eq:leq}. A similar result follows from \cite[Theorem 3.1]{RV}.

Now, a function $g\in V$ with $\Vert g\Vert\leq 1$ can be written uniquely as $g=\sum_{j=1}^dx_jb_j$ with $x\in\mathbb{R}^d$ and $\Vert x\Vert_2\leq 1$. Using this and  $\Vert g\Vert_n^2=x^TB_nx$, we obtain
\begin{equation}\label{eq:rerav2}
 \sup_{g\in V,\Vert g\Vert\leq 1}|\Vert g\Vert_n^2-\Vert g\Vert^2|=\sup_{x\in \mathbb{R}^d,\Vert x\Vert_2 \leq 1}|x^T(B_n-I)x|=\Vert B_n-I\Vert_{\operatorname{op}},
\end{equation}
where the latter equality follows from the spectral theorem. Moreover, we have that $\mathcal{E}_{\delta}$ holds if and only if 
\[\sup_{g\in V,\Vert g\Vert\leq 1}|\Vert g\Vert_n^2-\Vert g\Vert^2|\leq \delta.\]
Applying this, \eqref{eq:rerav}, and \eqref{eq:rerav2}, we complete the proof.
\end{proof}

\subsection{Analysis of the variance via Von Neumann's theorem} The basic theorem in the theory of projections on sumspaces is due to von Neumann \cite{vN}. We state the following version dealing only with the first component, which is a consequence of \cite[(15) on page 378]{Ar} (see also \cite[Theorem 2.C in Appendix A.4]{BKRW}). 
\begin{lemma}\label{Arvn} Let $\mathcal{H}_1$ and $\mathcal{H}_2$ be two closed subspaces of a Hilbert space $\mathcal{H}$ with inner product $\langle\cdot,\cdot\rangle$ and norm $\|\cdot\|$. Suppose that 
$\rho_0(\mathcal{H}_1,\mathcal{H}_2)<1$. Let $\Pi$, $\Pi_1$, $\Pi_2$ be the orthogonal projections on $\mathcal{H}_1+\mathcal{H}_2$, $\mathcal{H}_1$, $\mathcal{H}_2$, respectively. Let $h\in\mathcal{H}$, and let $(\Pi h)_1\in\mathcal{H}_1$, $(\Pi h)_2\in\mathcal{H}_2$ be the unique elements such that $\Pi h=(\Pi h)_1+(\Pi h)_2$. Then
\[
\big\Vert (\Pi h)_1-\big(\Pi_1-\sum_{j=1}^k(\Pi_1\Pi_2)^j(1-\Pi_1)\big)h\big\Vert\rightarrow 0
\]
as $k$ goes to infinity.
\end{lemma}
\begin{remark}
If we set $h_1^{(1)}=\Pi_1h$ and proceed iteratively by setting $h_2^{(m)}=\Pi_2(h-h_1^{(m)})$ and $h_1^{(m+1)}=\Pi_1(h-h_2^{(m)})$, $m\geq 1$, then Lemma \ref{Arvn} can be rewritten as $\|(\Pi h)_1-h_1^{(m)}\|\rightarrow 0$. This procedure is often called ``backfitting''.
\end{remark}
In this section, we apply Lemma \ref{Arvn} to the finite sample setting, using results from the previous section. Recall that $\hat{\Pi}_{V}$ is the orthogonal projection from $\mathbb{R}^n$ to the subspace $\lbrace (g(X^1),\dots,g(X^n))^T|g\in V\rbrace$, and that $\hat{\Pi}_{V_1}$, $\hat{\Pi}_{V_2}$, and $\hat{\Pi}_{W_1}$ are defined analogously (replace $V$ by $V_1$, $V_2$, and $W_1$, respectively). We first prove:
\begin{proposition}\label{backf1} Let Assumption \ref{compass} and \ref{angle} be satisfied. Let
\[\rho_{0,\delta}=1-\frac{(1-\delta)}{(1+\delta)}(1-\rho_0).\]
If $\mathcal{E}_\delta$ holds, then we have
\begin{align*}
 \mathbb{E}\left[ \Vert \hat{\Pi}_{W_1}(\hat{\Pi}_V\boldsymbol{\epsilon})_1\Vert_n^2|X^1,\dots,X^n\right]\leq\frac{\sigma^2\dim W_1}{n}+\frac{1}{1-\rho_{0,\delta}^2}\frac{\sigma^2\operatorname{tr}(\hat{\Pi}_{W_1}\hat{\Pi}_{V_2})}{n}.\\
\end{align*}
\end{proposition}

In the proof, we will need the following result:
\begin{lemma}\label{trlem} Let $A\in\mathbb{R}^{k_1\times k_2}$ and $B\in\mathbb{R}^{k_2\times k_1}$. Then
\begin{itemize}
\item[(i)] $\operatorname{tr}(AB)=\operatorname{tr}(BA)$.
\item[(ii)]  $|\operatorname{tr}(AB)|\leq \sqrt{\operatorname{tr}(AA^T)\operatorname{tr}(BB^T)}$.
\item[(iii)] Let $k_1=k_2$ and $B$ be symmetric and positive semi-definite. Then \[|\operatorname{tr}(AB)|\leq \Vert A\Vert_{\operatorname{op}}\operatorname{tr}(B),\]
where $\Vert A\Vert_{\operatorname{op}}=\sup_{\Vert x\Vert_2=1}\Vert Ax\Vert_2$ denotes the operator norm. Here, $\Vert\cdot\Vert_2$ denotes the Euclidean norm. 
\end{itemize}
\end{lemma}
For completeness, a proof of Lemma \ref{trlem} (iii) is given in Appendix \ref{app2h}.
\begin{proof}[Proof of Proposition \ref{backf1}.]
Throughout the proof, suppose that $\mathcal{E}_\delta$ holds. Furthermore, we consider $V$ as a subset of $\mathbb{R}^n$. This is no 
restriction, since (on $\mathcal{E}_\delta$) each element $g\in V$ is uniquely determined by $(g(X^1),\dots,g(X^n))^T$. 
From \eqref{eq:backfcond} and Lemma \ref{Arvn} applied to $(\mathcal{H},\langle \cdot,\cdot\rangle)=(V,\langle \cdot,\cdot\rangle_n)$ and $\mathcal{H}_j=V_j$, $j=1,2$, we have 
\begin{equation*}
\big\Vert (\hat{\Pi}_V\boldsymbol{\epsilon})_1-\big(\hat{\Pi}_{V_1}-\sum_{j=1}^k(\hat{\Pi}_{V_1}\hat{\Pi}_{V_2})^j(1-\hat{\Pi}_{V_1})\big)\boldsymbol{\epsilon}\big\Vert_n \rightarrow 0
\end{equation*}
as $k$ goes to infinity. 
From \eqref{eq:backfcond}, we have 
\begin{equation}\label{eq:pa1}
\Vert\hat{\Pi}_{V_1}g_2\Vert_n\leq\rho_{0,\delta}\Vert g_2\Vert_n
\end{equation}
 for all $g_2\in V_2$, which follows from 
\[\Vert\hat{\Pi}_{V_1}g_2\Vert_n^2=\langle \hat{\Pi}_{V_1}g_2,\hat{\Pi}_{V_1}g_2\rangle_n=\langle \hat{\Pi}_{V_1}g_2,g_2\rangle_n\leq \rho_{0,\delta}\Vert\hat{\Pi}_{V_1}g_2\Vert_n\Vert g_2\Vert_n.\] 
Similarly, we have 
\begin{equation}\label{eq:pa2}
\Vert\hat{\Pi}_{V_2}g_1\Vert_n\leq\rho_{0,\delta}\Vert g_1\Vert_n
\end{equation}
for all $g_1\in V_1$. 
This gives the improved convergence result
\begin{align*}
&\big\Vert (\hat{\Pi}_V\boldsymbol{\epsilon})_1-\big(\hat{\Pi}_{V_1}-\sum_{j=1}^k(\hat{\Pi}_{V_1}\hat{\Pi}_{V_2})^j(1-\hat{\Pi}_{V_1})\big)\boldsymbol{\epsilon}\big\Vert_n \\
&\leq\sum_{j=k+1}^\infty\Vert (\hat{\Pi}_{V_1}\hat{\Pi}_{V_2})^j(1-\hat{\Pi}_{V_1})\boldsymbol{\epsilon}\Vert_n\\
&\leq\sum_{j=k+1}^\infty\rho_{0,\delta}^{2j-1}\Vert \boldsymbol{\epsilon}\Vert_n\\&=\frac{\rho_{0,\delta}^{2k+1}}{1-\rho_{0,\delta}^2} \Vert \boldsymbol{\epsilon}\Vert_n,
\end{align*}
and also
\begin{align}
&\big\Vert \hat{\Pi}_{W_1}(\hat{\Pi}_V\boldsymbol{\epsilon})_1-\hat{\Pi}_{W_1}\big(\hat{\Pi}_{V_1}-\sum_{j=1}^k(\hat{\Pi}_{V_1}\hat{\Pi}_{V_2})^j(1-\hat{\Pi}_{V_1})\big)\boldsymbol{\epsilon}\big\Vert_n\nonumber\\
&\leq\frac{\rho_{0,\delta}^{2k+1}}{1-\rho_{0,\delta}^2} \Vert \boldsymbol{\epsilon}\Vert_n.\label{eq:awr}
\end{align}
Applying \eqref{eq:awr} and the bound $(x+y)^2\leq(1+\epsilon)x^2+(1+1/\epsilon)y^2$, $\epsilon>0$, and then taking expectation, we obtain
\begin{align}
& \mathbb{E}\left[ \Vert \hat{\Pi}_{W_1}(\hat{\Pi}_V\boldsymbol{\epsilon})_1\Vert_n^2\ |\ X^1,\dots,X^n\right]\nonumber \\
&\leq (1+\epsilon)\mathbb{E}\left[  \big\Vert\hat{\Pi}_{W_1}\big(\hat{\Pi}_{V_1}-\sum_{j=1}^k(\hat{\Pi}_{V_1}\hat{\Pi}_{V_2})^j(1-\hat{\Pi}_{V_1})\big)\boldsymbol{\epsilon}\big\Vert_n^2\ \big|\ X^1,\dots,X^n\right]\nonumber \\
 &+\left(1+1/\epsilon\right)\frac{\rho_{0,\delta}^{4k+2}}{(1-\rho_{0,\delta}^2)^2}\sigma^2 \label{eq:prere}.
\end{align}
Since $\mathbb{E}\left[ \Vert A\boldsymbol{\epsilon}\Vert_n^2\right]=\sigma^2\operatorname{tr}(AA^T)/n$ for all $A\in\mathbb{R}^{n\times n}$, it remains to bound the trace of 
\begin{equation}\label{eq:trof}
 \hat{\Pi}_{W_1}\big(\hat{\Pi}_{V_1}-\sum_{j=1}^k(\hat{\Pi}_{V_1}\hat{\Pi}_{V_2})^j(1-\hat{\Pi}_{V_1})\big)\Big(\hat{\Pi}_{V_1}-\sum_{j=1}^k(\hat{\Pi}_{V_1}\hat{\Pi}_{V_2})^j(1-\hat{\Pi}_{V_1})\Big)^T\hat{\Pi}_{W_1}.
\end{equation}
Using 
\begin{align*}
&\hat{\Pi}_{V_1}-\sum_{j=1}^k(\hat{\Pi}_{V_1}\hat{\Pi}_{V_2})^j(1-\hat{\Pi}_{V_1})\\
&=\sum_{j=0}^{k-1}\hat{\Pi}_{V_1}(\hat{\Pi}_{V_2}\hat{\Pi}_{V_1})^j(1-\hat{\Pi}_{V_2})+(\hat{\Pi}_{V_1}\hat{\Pi}_{V_2})^k\hat{\Pi}_{V_1},
\end{align*}
\eqref{eq:trof} is equal to
\begin{align*}
 &\hat{\Pi}_{W_1}\sum_{j=0}^k(\hat{\Pi}_{V_1}\hat{\Pi}_{V_2})^j\hat{\Pi}_{V_1}\Big(\hat{\Pi}_{V_1}-(1-\hat{\Pi}_{V_1})\sum_{j=1}^k(\hat{\Pi}_{V_2}\hat{\Pi}_{V_1})^j\Big)\hat{\Pi}_{W_1}\\
 -&\hat{\Pi}_{W_1}\sum_{j=1}^k(\hat{\Pi}_{V_1}\hat{\Pi}_{V_2})^j\Big((1-\hat{\Pi}_{V_2})\sum_{j=0}^{k-1}(\hat{\Pi}_{V_1}\hat{\Pi}_{V_2})^j\hat{\Pi}_{V_1}+(\hat{\Pi}_{V_1}\hat{\Pi}_{V_2})^k\hat{\Pi}_{V_1}\Big)\hat{\Pi}_{W_1}
\end{align*}
and, since $\hat{\Pi}_{V_1}(1-\hat{\Pi}_{V_1})=0$ and $\hat{\Pi}_{V_2}(1-\hat{\Pi}_{V_2})=0$, this is equal to
\begin{equation*}
 \hat{\Pi}_{W_1}\sum_{j=0}^k(\hat{\Pi}_{V_1}\hat{\Pi}_{V_2})^j\hat{\Pi}_{W_1}-\hat{\Pi}_{W_1}\sum_{j=1}^k(\hat{\Pi}_{V_1}\hat{\Pi}_{V_2})^{j+k}\hat{\Pi}_{W_1}.
\end{equation*}
By Lemma \ref{trlem} (i) and the identities $\hat{\Pi}_{W_1}\hat{\Pi}_{V_1}=\hat{\Pi}_{W_1}$ and $\hat{\Pi}_{V_2}\hat{\Pi}_{V_2}=\hat{\Pi}_{V_2}$, we have for $j=1,\dots,2k$,
\begin{align*}
\operatorname{tr}(\hat{\Pi}_{W_1}(\hat{\Pi}_{V_1}\hat{\Pi}_{V_2})^j\hat{\Pi}_{W_1})
&=\operatorname{tr}((\hat{\Pi}_{V_2}\hat{\Pi}_{V_1})^{j-1}\hat{\Pi}_{V_2}\hat{\Pi}_{W_1}\hat{\Pi}_{V_2})\\
&=\operatorname{tr}((\hat{\Pi}_{V_2}\hat{\Pi}_{V_1}\hat{\Pi}_{V_2})^{j-1}\hat{\Pi}_{V_2}\hat{\Pi}_{W_1}\hat{\Pi}_{V_2}).
\end{align*}
Thus the trace of \eqref{eq:trof} is bounded by
\begin{equation}\label{eq:ikl}
\dim W_1+\sum_{j=1}^{2k}|\operatorname{tr}((\hat{\Pi}_{V_2}\hat{\Pi}_{V_1}\hat{\Pi}_{V_2})^{j-1}\hat{\Pi}_{V_2}\hat{\Pi}_{W_1}\hat{\Pi}_{V_2})|.
\end{equation}
Applying Lemma \ref{trlem} (iii), this can be bounded by
\begin{equation*}
\dim W_1+\sum_{j=0}^{2k-1}\Vert\hat{\Pi}_{V_2}\hat{\Pi}_{V_1}\hat{\Pi}_{V_2} \Vert_{\operatorname{op}}^j\operatorname{tr}(\hat{\Pi}_{V_2}\hat{\Pi}_{W_1}\hat{\Pi}_{V_2}).
\end{equation*}
By \eqref{eq:pa1} and \eqref{eq:pa2}, we have
\begin{equation}\label{eq:pacor}
\Vert \hat{\Pi}_{V_2}\hat{\Pi}_{V_1}\hat{\Pi}_{V_2}\Vert_{\operatorname{op}}\leq \rho^2_{0,\delta}.
\end{equation}
Moreover, we have $\operatorname{tr}(\hat{\Pi}_{V_2}\hat{\Pi}_{W_1}\hat{\Pi}_{V_2})=\operatorname{tr}(\hat{\Pi}_{W_1}\hat{\Pi}_{V_2})$.
Thus \eqref{eq:ikl} is bounded by
\begin{equation}\label{eq:prere2}
\dim W_1+\frac{1}{1-\rho_{0,\delta}^2}\operatorname{tr}(\hat{\Pi}_{W_1}\hat{\Pi}_{V_2}).
\end{equation}
From \eqref{eq:prere}-\eqref{eq:prere2}, we conclude that 
\begin{align*}
 &\mathbb{E}\left[ \Vert \hat{\Pi}_{W_1}(\hat{\Pi}_V\boldsymbol{\epsilon})_1\Vert_n^2|X^1,\dots,X^n\right]\\
 &\leq(1+\epsilon)\left(\frac{\sigma^2\dim W_1}{n}+\frac{1}{1-\rho_{0,\delta}^2}\frac{\sigma^2\operatorname{tr}(\hat{\Pi}_{W_1}\hat{\Pi}_{V_2})}{n}\right)\\
 &+\left(1+1/\epsilon\right)\frac{\rho_{0,\delta}^{4k+2}}{(1-\rho_{0,\delta}^2)^2}\sigma^2. 
\end{align*}
Now, send $k$ off to infinity first, and then let $\epsilon$ go to zero. This completes the proof.
\end{proof}
From Proposition \ref{backf1}, we obtain a first upper bound for the variance term, which does not depend on the dimension of $V_2$. Note that in Proposition \ref{backf2} and Corollary \ref{corbackf2}, we show that this upper bound can be further refined.
\begin{corollary}\label{corbackf}  Let Assumption \ref{compass} and \ref{angle} be satisfied. If $\mathcal{E}_\delta$ holds, then we have
\[
 \mathbb{E}\left[ \Vert \hat{\Pi}_{W_1}(\hat{\Pi}_V\boldsymbol{\epsilon})_1\Vert_n^2|X^1,\dots,X^n\right]\leq \frac{1+\delta}{1-\delta}\frac{1}{1-\rho_0^2}\frac{\sigma^2\dim W_1}{n}.
\]
\end{corollary}

\begin{proof}
By Lemma \ref{trlem} (i) and (ii), we have \[\operatorname{tr}(\hat{\Pi}_{W_1}\hat{\Pi}_{V_2})=\operatorname{tr}(\hat{\Pi}_{W_1}\hat{\Pi}_{V_2}\hat{\Pi}_{W_1}\hat{\Pi}_{W_1}).\] Applying Lemma \ref{trlem} (iii) and \eqref{eq:pacor}, we obtain on $\mathcal{E}_\delta$,
\[
\operatorname{tr}(\hat{\Pi}_{W_1}\hat{\Pi}_{V_2})\leq \Vert \hat{\Pi}_{W_1}\hat{\Pi}_{V_2}\hat{\Pi}_{W_1}\Vert_{\operatorname{op}}\operatorname{tr}(\hat{\Pi}_{W_1})\leq \rho_{0,\delta}^2\dim W_1.
\]
Thus, if $\mathcal{E}_\delta$ holds, Proposition \ref{backf1} yields
\[
\mathbb{E}\left[ \Vert \hat{\Pi}_{W_1}(\hat{\Pi}_V\boldsymbol{\epsilon})_1\Vert_n^2|X^1,\dots,X^n\right]\leq \frac{1}{1-\rho_{0,\delta}^2}\frac{\sigma^2\dim W_1}{n}.
\]
Since $(1-c(1-\varrho))^2\leq 1-c(1-\varrho^2)$ for $\varrho\in[0,1]$ and a constant $0\leq c\leq 1$ (both functions are equal to $1$ at the right endpoint $\varrho=1$ and the derivative of the left hand side is greater or equal than the derivative of the right hand side for all $\varrho\in[0,1]$), we obtain (set $c=(1-\delta)/(1+\delta)$ and $\varrho=\rho_0$)
\begin{equation}\label{eq:prere3}
\frac{1}{1-\rho_{0,\delta}^2}\leq\frac{1+\delta}{1-\delta}\frac{1}{1-\rho_0^2}.
\end{equation}
This completes the proof.
\end{proof}
\begin{proposition}\label{backf2} Let Assumption \ref{compass}, \ref{angle}, and \ref{assinfty} be satisfied. Then
\begin{equation*} 
 \frac{1}{n}\mathbb{E}\left[1_{\mathcal{E}_\delta}\operatorname{tr}(\hat{\Pi}_{W_1}\hat{\Pi}_{V_2})\right]\leq\frac{1}{(1-\delta)^2}\left(\frac{\left\|\Pi_{V_2}|_{W_1}\right\|^2_{HS}}{n}+\frac{\dim W_1}{n}\frac{\varphi^2d}{n}\right).
\end{equation*}
\end{proposition}
\begin{remark}
By considering $\Pi_{W_1}\Pi_{V_2}\Pi_{W_1}$ as a map from $W_1\subseteq L^2(\mathbb{P}^{X_1})$ to itself, we have $\left\|\Pi_{V_2}|_{W_1}\right\|^2_{HS}=\operatorname{tr}(\Pi_{W_1}\Pi_{V_2}\Pi_{W_1})$.
\end{remark}
\begin{proof}
Let $\left\{\phi_{1j}\right\}_{1\leq j \leq \dim W_1}$ be an orthonormal basis of $W_1$, and let $\left\{\phi_{2j}\right\}_{1\leq j \leq d_2}$ be an orthonormal basis of $V_2$. Let
\[
Z_1=(\phi_{1j}(X_{1}^i))_{1\leq i\leq n,1\leq j\leq \dim W_1}
\]
and
\[
Z_2=(\phi_{2j}(X_{2}^i))_{1\leq i\leq n,1\leq j\leq d_2},
\]
Now, suppose that $\mathcal{E}_\delta$ holds. Then we have $\hat{\Pi}_{W_1}=Z_1(Z_1^TZ_1)^{-1}Z_1^T$ and $\hat{\Pi}_{V_2}=Z_2(Z_2^TZ_2)^{-1}Z_2^T$. Thus
\begin{equation*}
\operatorname{tr}(\hat{\Pi}_{W_1}\hat{\Pi}_{V_2})=\operatorname{tr}\left(\left(\frac{1}{n}Z_1^TZ_1\right)^{-1}\frac{1}{n}Z_1^TZ_2\left(\frac{1}{n}Z_2^TZ_2\right)^{-1}\frac{1}{n}Z_2^TZ_1\right),
\end{equation*}
where we applied Lemma \ref{trlem} (i). By Theorem \ref{mfg}, we have $\Vert (1/n)Z_j^TZ_j-I\Vert_{\operatorname{op}}\leq \delta$ and thus 
\[\left(\frac{1}{n}Z_j^TZ_j\right)^{-1}=\sum_{k\geq 0}\left(I-\frac{1}{n}Z_j^TZ_j\right)^k\]
for $j=1,2$. We conclude that
\begin{align*}
&\mathbb{E}\left[1_{\mathcal{E}_\delta}\operatorname{tr}(\hat{\Pi}_{W_1}\hat{\Pi}_{V_2})\right]\\
& \leq\sum_{k,l=0}^\infty\mathbb{E}\left[1_{\mathcal{E}_\delta}\left|\operatorname{tr}\left(\left(I-\frac{1}{n}Z_1^TZ_1\right)^k\frac{1}{n}Z_1^TZ_2\left(I-\frac{1}{n}Z_2^TZ_2\right)^{l}\frac{1}{n}Z_2^TZ_1\right)\right|\right]\\
&\leq \sum_{k,l=0}^\infty\mathbb{E}\left[1_{\mathcal{E}_\delta}\sqrt{\operatorname{tr}\left(\left(I-\frac{1}{n}Z_1^TZ_1\right)^{2k}\frac{1}{n}Z_1^TZ_2\left(\frac{1}{n}Z_1^TZ_2\right)^T\right)}\right.\\
&\hspace{20mm}\left.\cdot\sqrt{\operatorname{tr}\left(\left(I-\frac{1}{n}Z_2^TZ_2\right)^{2k}\frac{1}{n}Z_2^TZ_1\left(\frac{1}{n}Z_2^TZ_1\right)^T\right)}\right]\\
&\leq \sum_{k,l=0}^\infty\delta^{k+l}\mathbb{E}\left[\operatorname{tr}\left(\frac{1}{n}Z_1^TZ_2\left(\frac{1}{n}Z_1^TZ_2\right)^T\right)\right],
\end{align*}
where we applied Lemma \ref{trlem} (i) and (ii) in the second inequality and Lemma \ref{trlem} (i) and (iii) and the bound $\Vert (1/n)Z_j^TZ_j-I\Vert_{\operatorname{op}}\leq \delta$ in the third inequality. Now
\begin{align*}
&\mathbb{E}\left[ \operatorname{tr}\left(\frac{1}{n}Z_1^TZ_2\left(\frac{1}{n}Z_1^TZ_2\right)^T\right)\right] \\
&=\sum_{j=1}^{\dim W_1}\sum_{k=1}^{d_2}\left(\left(\mathbb{E}\left[\phi_{1j}(X_1)\phi_{2k}(X_2)\right]\right) ^2 +\frac{1}{n}\operatorname{Var}(\phi_{1j}(X_1)\phi_{2k}(X_2))\right)\\
&\leq\sum_{j=1}^{\dim W_1}\sum_{k=1}^{d_2}\left\langle\phi_{1j},\phi_{2k} \right\rangle^2+\dim W_1\frac{\varphi^2d}{n},
\end{align*}
where we applied \eqref{eq:leq}. Finally, we use
\[\left\|\Pi_{V_2}|_{W_1}\right\|^2_{HS}=\sum_{j=1}^{\dim W_1}\|\Pi_{V_2}\phi_{1j}\|^2=\sum_{j=1}^{\dim W_1}\sum_{k=1}^{d_2}\left\langle\phi_{1j},\phi_{2k} \right\rangle^2\]
This completes the proof.
\end{proof}
Combining Proposition \ref{backf1} and \ref{backf2}, we obtain the following improvement of Corollary \ref{corbackf}:
\begin{corollary}\label{corbackf2} Let Assumption \ref{compass}, \ref{angle}, and \ref{assinfty} be satisfied. Then
\begin{align*}
 &\mathbb{E}\left[ 1_{\mathcal{E}_\delta}\Vert \hat{\Pi}_{W_1}(\hat{\Pi}_V\boldsymbol{\epsilon})_1\Vert_n^2\right]\\&\leq \frac{\sigma^2\dim W_1}{n}+\frac{(1+\delta)}{(1-\delta)^3}\frac{1}{1-\rho_0^2}\left(\frac{\sigma^2\left\|\Pi_{V_2}|_{W_1}\right\|^2_{HS}}{n}+\frac{\sigma^2\dim W_1}{n}\frac{\varphi^2d}{n}\right).
\end{align*}
\end{corollary}
\subsection{End of the proof of Theorem \ref{thm1}}  
Applying the arguments of \cite{B}, we obtain
\begin{equation}\label{eq:pr1}
{E}\left[ \Vert f_1-\hat{f}_{1}^*\Vert^2\right]\leq\mathbb{E}
\left[ 1_{\mathcal{E}_\delta}\Vert f_1-(\hat{f}_V)_1\Vert^2\right]+R_n.
\end{equation}
The details can be found in Appendix \ref{app2ha}. By the projection theorem (see, e.g., \cite[Theorem II.3]{RS}), we have
\begin{equation}\label{eq:pr2}
\Vert f_1-(\hat{f}_V)_1\Vert^2=\Vert f_1-\Pi_{V_1}f_1\Vert^2
+\Vert \Pi_{V_1}f_1-(\hat{f}_V)_1\Vert^2.
\end{equation}
By the definition of $\mathcal{E}_\delta$ and the definition of $\hat{f}_V$, we have
\begin{align}
&\mathbb{E}\left[ 1_{\mathcal{E}_\delta}\Vert \Pi_{V_1}f_1-(\hat{f}_V)_1\Vert^2\right]\nonumber\\
 &\leq\frac{1}{(1-\delta)}\mathbb{E}\left[ 1_{\mathcal{E}_\delta}\Vert \Pi_{V_1}f_1-(\hat{f}_V)_1\Vert_n^2\right]\nonumber\\
 &=\frac{1}{(1-\delta)}\mathbb{E}\left[ 1_{\mathcal{E}_\delta}\Vert \Pi_{V_1}f_1-(\hat{\Pi}_V\mathbf{Y})_1\Vert_n^2\right]\nonumber\\
 &=\frac{1}{(1-\delta)}\mathbb{E}\left[ 1_{\mathcal{E}_\delta}
 \mathbb{E}\left[ \Vert \Pi_{V_1}f_1-(\hat{\Pi}_V\mathbf{Y})_1\Vert_n^2|X^1,\dots,X^n\right] \right].\label{eq:ggw}
\end{align}
Recall from Section \ref{es} that on $\mathcal{E}_\delta$ each $g\in V$ is determined uniquely by $(g(X^1),\dots,g(X^n))^T$, which implies that on $\mathcal{E}_\delta$ we don't have to distinguish between those objects. For instance, if $\mathcal{E}_\delta$ holds, then $\hat{\Pi}_V$ and $\hat{\Pi}_{V_1}$ can also be seen as maps from $L^2(\mathbb{P}^X)$ to $V$ and $V_1$, respectively (by letting $\hat{\Pi}_Vh$ (resp. $\hat{\Pi}_{V_1}h$) equal to $\hat{\Pi}_V(h(X^1),\dots,h(X^n))^T$ (resp. $\hat{\Pi}_{V_1}(h(X^1),\dots,h(X^n))^T$) for $h\in L^2(\mathbb{P}^X)$).
Moreover, if $\mathcal{E}_\delta$ holds, then we have $\mathbb{E}[(\hat{\Pi}_V\boldsymbol{\epsilon})_1|X^1,\dots,X^n]=0$ and $(\hat{\Pi}_V\mathbf{Y})_1=(\hat{\Pi}_Vf)_1+(\hat{\Pi}_V\boldsymbol{\epsilon})_1$ (the former follows for instance from \eqref{eq:awr}). Thus \eqref{eq:ggw} is equal to 
\begin{equation}
\frac{1}{(1-\delta)}\mathbb{E}\left[ 1_{\mathcal{E}_\delta}\left(\Vert \Pi_{V_1}f_1-(\hat{\Pi}_Vf)_1\Vert_n^2+
 \mathbb{E}\left[\Vert (\hat{\Pi}_V\boldsymbol{\epsilon})_1\Vert_n^2|X_1,\dots,X_n\right] \right) \right].\label{eq:pr3}
\end{equation}
From \eqref{eq:pr2}-\eqref{eq:pr3} and the definition of $\mathcal{E}_\delta$, we obtain
\begin{align}
 \mathbb{E}&\left[ 1_{\mathcal{E}_\delta}\Vert f_1-(\hat{f}_V)_1\Vert^2\right]\nonumber\\
  &\leq\frac{(1+\delta)}{(1-\delta)}\mathbb{E}\left[ 1_{\mathcal{E}_\delta}\Big(\Vert f_1-\Pi_{V_1}f_1\Vert^2+\Vert \Pi_{V_1}f_1-(\hat{\Pi}_Vf)_1\Vert^2\Big)\right] \nonumber\\
  &+\frac{1}{(1-\delta)}\mathbb{E}\left[ 1_{\mathcal{E}_\delta} \mathbb{E}\left[\Vert (\hat{\Pi}_V\boldsymbol{\epsilon})_1\Vert_n^2|X^1,\dots,X^n\right] \right].\label{eq:pr4}
\end{align}
Applying the projection theorem and Corollary \ref{corbackf}, this is bounded by
\begin{equation*}
\frac{(1+\delta)}{(1-\delta)}\mathbb{E}\left[ 1_{\mathcal{E}_\delta}\Vert f_1-(\hat{\Pi}_Vf)_1\Vert^2\right]
 +\frac{(1+\delta)}{(1-\delta)^2}\frac{1}{1-\rho_0^2}
 \frac{\sigma^2\dim V_1}{n}.
\end{equation*}
By Assumption \ref{angle} and Lemma \ref{angleequiv}, we have
\[\Vert f_1-(\hat{\Pi}_Vf)_1\Vert^2\leq \frac{1}{1-\rho_0^2}\Vert f-\hat{\Pi}_Vf\Vert^2.\]
The projection theorem implies that
\[\Vert f-\hat{\Pi}_Vf\Vert^2=\Vert f-\Pi_Vf\Vert^2+\Vert\Pi_Vf-\hat{\Pi}_Vf\Vert^2.\]
Now, the following proposition completes the proof.
\begin{lemma}\label{infty}
 Let Assumption \ref{assinfty} be satisfied. Then 
\begin{equation*}
\mathbb{E}\left[ 1_ {\mathcal{E}_\delta}\Vert ( \Pi_V-\hat{\Pi}_V)f\Vert^2 \right]\leq  \frac{1}{(1-\delta)^2}\frac{\varphi^2d}{n}\Vert f-\Pi_Vf\Vert^2.
\end{equation*}
\end{lemma}
\begin{remark}
Instead of applying Lemma \ref{infty}, one can also apply the easier and weaker bound
\[\Vert\mathrm{\Pi}_Vf-\hat{\mathrm{\Pi}}_Vf\Vert^2\leq \frac{1}{1-\delta}\Vert\mathrm{\Pi}_Vf-\hat{\mathrm{\Pi}}_Vf\Vert^2_n\leq
\frac{1}{1-\delta}\Vert f-\mathrm{\Pi}_Vf\Vert^2_n,\]
which follows from the definition of $\mathcal{E}_\delta$ and the projection theorem. 
\end{remark}
\begin{proof}[Proof of Lemma \ref{infty}] Throughout the proof suppose that the event $\mathcal{E}_\delta$ holds. 
Let $b_1,\dots,b_{d}$ be an orthonormal basis of $V$. Let
\[B_n=(\langle b_j,b_k\rangle_n)_{1\leq j,k\leq d},\]
and let
\[x=(\langle b_1,f\rangle,\dots,\langle b_d,f\rangle)^T\ \text{ and }\ x_n=(\langle b_1,f\rangle_n,\dots,\langle b_d,f\rangle_n)^T.\]
Then we have
\[\Pi_V f=\sum_{j=1}^d x_jb_j\ \text{ and }\ \hat{\Pi}_Vf=\sum_{j=1}^d(B_n^{-1}x_n)_jb_j\]
and thus
\begin{equation}\label{eq:fireq}
\Vert  (\Pi_V-\hat{\Pi}_V)f\Vert^2=\Vert B_n^{-1}x_n-x\Vert_2^2.
\end{equation}
Since  $\mathcal{E}_\delta$ holds, we have $\Vert B_n-I\Vert_{\operatorname{op}}\leq \delta$ (see the proof of Theorem \ref{mfg}). This implies that
\begin{equation*}
B_n^{-1}=\sum_{k\geq 0}(I-B_n)^k.
\end{equation*}
Thus
\begin{align*}
B_n^{-1}x_n-x&=\sum_{k\geq 0}(I-B_n)^kx_n-x\\
&=\sum_{k\geq 0}(I-B_n)^k(x_n-x)+\sum_{k\geq 1}(I-B_n)^kx.
\end{align*}
Applying the bounds $\Vert B_n-I\Vert_{\operatorname{op}}\leq \delta$ and $(x+y)^2\leq (1+\epsilon)x^2+(1+1/\epsilon)y^2$, $\epsilon>0$ arbitrary, we obtain
\begin{align}
  \Vert B_n^{-1}x_n-x\Vert_2^2&\leq \frac{1}{(1-\delta)^2}(\Vert x_n-x\Vert_2+\Vert (B_n-I)x\Vert_2)^2\nonumber\\
  &\leq \frac{1}{(1-\delta)^2} ((1+\epsilon)\Vert x_n-x\Vert_2^2+(1+1/\epsilon) \Vert (B_n-I)x\Vert_2^2).\label{eq:preq1}
\end{align}
Now we have
\begin{align}
 \mathbb{E}\left[ \Vert x_n-x\Vert_2^2\right]&=\mathbb{E}\Big[\sum_{j=1}^{d}(\langle b_j,f\rangle_n-\langle b_j,f\rangle)^2 \Big] \nonumber\\
 &=\frac{1}{n}\sum_{j=1}^{d}\operatorname{Var}(b_j(X)f(X))\leq \frac{\varphi^2d}{n}\Vert f\Vert^2,\label{eq:preq2}
\end{align}
where we applied \eqref{eq:leq}. The $j$th coordinate of $(B_n-I)x$ is is equal to 
\begin{equation}\label{eq:hhh}
\langle b_j,\sum_{k=1}^d b_kx_k\rangle_n-\langle b_j,f\rangle=\langle b_j,\Pi_V f\rangle_n-\langle b_j,\Pi_V f\rangle
\end{equation} 
and thus
\begin{align}
 \mathbb{E}\left[ \Vert (B_n-I)x\Vert_2^2\right]&=\mathbb{E}\Big[\sum_{j=1}^{d}(\langle b_j,\Pi_V f\rangle_n-\langle b_j,\Pi_V f\rangle)^2 \Big] \nonumber\\&=\frac{1}{n}\sum_{j=1}^{d_1}\operatorname{Var}(b_j(X)\Pi_Vf(X))\leq \frac{\varphi^2d}{n}\Vert \Pi_Vf\Vert^2.\label{eq:preq3}
\end{align}
Applying \eqref{eq:fireq}-\eqref{eq:preq3}, we conclude that
\begin{equation}\label{eq:inftyeq}
\mathbb{E}\left[ 1_{\mathcal{E}_\delta}\Vert ( \Pi_V-\hat{\Pi}_V)f\Vert^2 \right]\leq 
\frac{1}{(1-\delta)^2} \frac{\varphi^2d}{n}((1+\epsilon)\Vert f\Vert^2+(1+1/\epsilon) \Vert \Pi_Vf\Vert^2).
\end{equation}
Finally, since $\Pi_V$ and $\hat{\Pi}_V$ fix elements in $V$, we obtain $(\Pi_V-\hat{\Pi}_V)\Pi_Vf=0$ and thus $(\Pi_V-\hat{\Pi}_V)f=(\Pi_V-\hat{\Pi}_V)(1-\Pi_V)f$. Combining this with \eqref{eq:inftyeq}, we see that
\begin{equation*}
\mathbb{E}\left[ 1_{\mathcal{E}_\delta}\Vert ( \Pi_V-\hat{\Pi}_V)f\Vert^2 \right]\leq 
\frac{1}{(1-\delta)^2} \frac{\varphi^2d}{n}(1+\epsilon)\Vert f-\Pi_Vf\Vert^2.
\end{equation*}
Since $\epsilon>0$ is arbitrary, this completes the proof.
\end{proof}
\subsection{End of the proof of Theorem \ref{thm2}} Compared with the previous section, we modify our analysis of the bias term, which is based the following two lemmas. Moreover, we replace Corollary \ref{corbackf2} by Corollary~\ref{corbackf}.
\begin{lemma}\label{lb1} Let Assumption \ref{compass} and \ref{angle} be satisfied. If $\mathcal{E_\delta}$ holds, then we have
\begin{equation*}
\| \hat{\Pi}_{V_1}h_1-(\hat{\Pi}_V  h_1)_1\|_n^2\leq \frac{(1+\delta)}{(1-\delta)}\frac{1}{(1-\rho_{0}^2)}\|h_1-\hat{\Pi}_{V_1}h_1\|_n^2
\end{equation*}
for all $h_1\in H_1$ and
\begin{equation*}
\|(\hat{\Pi}_Vh_2)_1\|_n^2\leq \frac{(1+\delta)}{(1-\delta)}\frac{1}{(1-\rho_{0}^2)}\|h_2-\hat{\Pi}_{V_2}h_2\|_n^2
\end{equation*}
for all $h_2\in H_2$. 
\end{lemma}
\begin{proof}[Proof of Lemma \ref{lb1}]
By \eqref{eq:empangs} and the fact that projections lower the norm, we have
\begin{align*}
\|\hat{\Pi}_{V_1}h_1-(\hat{\Pi}_V  h_1)_1\|_n^2&=\|(\hat{\Pi}_V  (h_1-\hat{\Pi}_{V_1}h_1))_1\|_n^2\\
&\leq \frac{(1+\delta)}{(1-\delta)}\frac{1}{(1-\rho_{0}^2)}\|h_1-\hat{\Pi}_{V_1}h_1\|_n^2.
\end{align*}
This gives the first inequality. Since $\hat{\Pi}_{V}\hat{\Pi}_{V_2}h_2=\hat{\Pi}_{V_2}h_2$ and $(\hat{\Pi}_{V_2}h_2)_1=0$, we have $(\hat{\Pi}_Vh_2)_1=(\hat{\Pi}_V(h_2-\hat{\Pi}_{V_2}h_2))_1$. Using the previous arguments, we conclude that
\[\|(\hat{\Pi}_Vh_2)_1\|_n^2=\|(\hat{\Pi}_V(h_2-\hat{\Pi}_{V_2}h_2))_1\|_n^2\leq \frac{(1+\delta)}{(1-\delta)}\frac{1}{(1-\rho_{0}^2)}\|h_2-\hat{\Pi}_{V_2}h_2\|_n^2.\]
This completes the proof. 
\end{proof}
\begin{lemma}\label{lb2} Let Assumption \ref{compass} and \ref{angle} be satisfied. Then we have
\begin{align*}
&\mathbb{E}\left[ 1_{\mathcal{E}_\delta}\|f_1-(\hat{\Pi}_Vf)_1\|_n^2\right]\\&\leq \frac{(1+\delta)}{(1-\delta)}\frac{3}{(1-\rho_{0}^2)}\left(\|f_1-\Pi_{V_1}f_1\|^2+\|f_2-\Pi_{V_2}f_2\|^2\right).
\end{align*}
\end{lemma}
\begin{proof}[Proof of Lemma \ref{lb2}]
By the projection theorem, we have
\[\|f_1-(\hat{\Pi}_V  f)_1\|_n^2=\|f_1-\hat{\Pi}_{V_1}f_1\|_n^2+\|\hat{\Pi}_{V_1}f_1-(\hat{\Pi}_V  f_1)_1-(\hat{\Pi}_V  f_2)_1\|_n^2.\]
Applying this, the bound $\|x+y\|^2\leq 2\|x\|^2+2\|y\|^2$, and Lemma \ref{lb2}, we obtain
\begin{align*}
&\mathbb{E}\left[ 1_{\mathcal{E}_\delta}\|f_1-(\hat{\Pi}_Vf)_1\|_n^2\right]\\&\leq\frac{(1+\delta)}{(1-\delta)}\frac{3}{(1-\rho_{0}^2)}\left(\mathbb{E}\left[\|f_1-\hat{\Pi}_{V_1}f_1\|_n^2+\|f_2-\hat{\Pi}_{V_2}f_2\|_n^2\right]\right).
\end{align*}
By the projection theorem, we have for $j=1,2$,
\[\|f_j-\hat{\Pi}_{V_j}f_j\|_n^2\leq\|f_j-\Pi_{V_j}f_j\|_n^2.\] 
Moreover, taking expectation, we get for $j=1,2$,
\[\mathbb{E}\left[\|f_j-\Pi_{V_j}f_j\|_n^2\right]=\|f_j-\Pi_{V_j}f_j\|^2.\] 
This completes the proof.
\end{proof}
Now, we begin with the proof of Theorem \ref{thm2}.
Repeating the steps \eqref{eq:pr1}-\eqref{eq:pr3} in the proof of Theorem \ref{thm1}, we have
\begin{align*}
\mathbb{E}&\left[ \Vert f_1-\hat{f}_{1}^*\Vert^2 \right]\nonumber\\
&\leq \Vert f_1-\Pi_{W_1}f_1\Vert^2+
\frac{1}{(1-\delta)} \mathbb{E}\left[1_{\mathcal{E}_\delta}\Vert \Pi_{W_1}f_1-\hat{\Pi}_{W_1}(\hat{\Pi}_Vf)_1\Vert_n^2 \right]\nonumber\\
&+\frac{1}{(1-\delta)} \mathbb{E}\left[ 1_{\mathcal{E}_\delta}\mathbb{E}\left[ \Vert \hat{\Pi}_{W_1}(\hat{\Pi}_V\boldsymbol{\epsilon})_1\Vert_n^2|X^1,\dots,X^n \right]\right]+R_n.
\end{align*}
Applying the bound $\|x+y\|^2\leq 2\|x\|^2+2\|y\|^2$ and the fact that projections lower the norm, we obtain
\begin{align*}
&\Vert \Pi_{W_1}f_1-\hat{\Pi}_{W_1}(\hat{\Pi}_Vf)_1\Vert_n^2\nonumber\\&\leq 2\Vert\Pi_{W_1}f_1-\hat{\Pi}_{W_1}f_1\Vert_n^2+2\Vert \hat{\Pi}_{W_1}f_1-\hat{\Pi}_{W_1}(\hat{\Pi}_Vf)_1\Vert_n^2\nonumber\\
&\leq 2\Vert\Pi_{W_1}f_1-\hat{\Pi}_{W_1}f_1\Vert_n^2+2\Vert f_1-(\hat{\Pi}_Vf)_1\Vert_n^2.
\end{align*}
Thus
\begin{align}
\mathbb{E}&\left[ \Vert f_1-\hat{f}_{1}^*\Vert^2 \right]\nonumber\\
&\leq \Vert f_1-\Pi_{W_1}f_1\Vert^2+
\frac{2}{(1-\delta)} \mathbb{E}\left[1_{\mathcal{E}_\delta}\Vert\Pi_{W_1}f_1-\hat{\Pi}_{W_1}f_1\Vert_n^2\right]\nonumber\\
&+\frac{1}{(1-\delta)} \mathbb{E}\left[ 1_{\mathcal{E}_\delta}\mathbb{E}\left[ \Vert \hat{\Pi}_{W_1}(\hat{\Pi}_V\boldsymbol{\epsilon})_1\Vert_n^2|X^1,\dots,X^n \right]\right]\nonumber\\
&+\frac{2}{(1-\delta)}\mathbb{E}\left[1_{\mathcal{E}_\delta}\Vert f_1-(\hat{\Pi}_Vf)_1\Vert_n^2\right]+R_n.\label{eq:fbiaeq}
\end{align}
Similarly as in Lemma \ref{infty}, we have 
\begin{equation}\label{eq:bieq2}
\mathbb{E}\left[ 1_ {\mathcal{E}_\delta}\Vert \Pi_{W_1}f_1-\hat{\Pi}_{W_1}f_1\Vert^2 \right]\leq  \frac{1}{(1-\delta)^2}\frac{\varphi^2d}{n}\Vert f_1-\Pi_{W_1}f_1\Vert^2.
\end{equation}
Inserting  \eqref{eq:bieq2}, Lemma \ref{lb2}, and Corollary \ref{corbackf2} into \eqref{eq:fbiaeq}, we complete the proof.\qed

\section{Proof of Theorem \ref{thm3} and \ref{thm4}}\label{prthm34}
\subsection{The bias term revisited}
In this subsection, we prove two results which lead to improvements of the bias term considered in Lemma \ref{lb2}. This improvement is possible under the additional regularity conditions on the design densities, namely \eqref{eq:condjdens} or \eqref{eq:condjdenspr}. We first prove:
\begin{proposition}\label{abrut} Let Assumption \ref{compass}, \ref{angle}, and \ref{HSce} be satisfied. Let $r$, $\phi(V_2)$, $\psi_\Pi(V_2)$, and $h_1$ be as in Theorem \ref{thm4}. Then
\begin{align}
 \Vert(\Pi_Vf_2)_1\Vert&\leq\frac{1}{1-\rho_0^2}\Vert h_1\Vert\psi_\Pi(V_2)\phi(V_2).\label{eq:abrutc}
\end{align}
\end{proposition}


\begin{proof} Let $\phi_1,\dots,\phi_{d_1}$ be an orthonormal basis of $V_1$. We have
\begin{align*}
&\Vert\Pi_{V_1}(f_2-\Pi_{V_2} f_2)\Vert^2\\
&=\sum_{j=1}^{d_1}\Big(\int_{S_1\times S_2}\phi_j(x_1)(f_2(x_2)-\Pi_{V_2}f_2(x_2))p(x_1,x_2)d(\nu_1\otimes\nu_2)(x_1,x_2) \Big)^2\\
&=\sum_{j=1}^{d_1}\Big(\int_{S_1}\phi_j(x_1)\int_{S_2}((1-\Pi_{V_2})f_2(x_2))\frac{p(x_1,x_2)}{p_1(x_1)p_2(x_2)}d\mathbb{P}^{X_2}(x_2)d\mathbb{P}^{X_1}(x_1) \Big)^2\\
&=\sum_{j=1}^{d_1}\Big(\int_{S_1} \Big\langle (1-\Pi_{V_2})f_2,\frac{p(x_1,\cdot)}{p_1(x_1)p_2(\cdot)}\Big\rangle_{L^2(\mathbb{P}^{X_2})} \phi_j(x_1)d\mathbb{P}^{X_1}(x_1) \Big)^2,
\end{align*}
where we already applied Fubini's theorem and Assumption \ref{HS} in the second equality.
Since orthogonal projections are idempotent and self-adjoint, the above is equal to
\begin{align*}
&=\sum_{j=1}^{d_1}\Big(\int_{S_1} \langle (1-\Pi_{V_2})f_2,(1-\Pi_{V_2})(r(x_1,\cdot))\rangle_{L^2(\mathbb{P}^{X_2})} \phi_j(x_1)d\mathbb{P}^{X_1}(x_1) \Big)^2 \\
&\leq\int_{S_1}  \langle (1-\Pi_{V_2})f_2,(1-\Pi_{V_2})(r(x_1,\cdot))\rangle_{L^2(\mathbb{P}^{X_2})}^2d\mathbb{P}^{X_1}(x_1)\\
&\leq \Vert h_1\Vert^2(\psi_\Pi(V_2)\phi(V_2))^2,
\end{align*}
where we applied Bessel's inequality in the first inequality and the Cauchy-Schwarz inequality and \eqref{eq:condjdenspr} in the second inequality. Thus we have shown that
\begin{equation}\label{eq:smarg}
 \Vert\Pi_{V_1}(f_2-\Pi_{V_2} f_2)\Vert \leq\Vert h_1\Vert\psi_\Pi(V_2)\phi(V_2).
\end{equation}
Now, by Lemma \ref{Arvn}, we have
\begin{equation}\label{eq:arpcon}
 \big\Vert (\Pi_V h)_1-\big(\Pi_{V_1}-\sum_{j=1}^k(\Pi_{V_1}\Pi_{V_2})^j(1-\Pi_{V_1})\big)h\big\Vert\rightarrow 0,
\end{equation}
as $k\rightarrow\infty$, for all $h\in L^2(\mathbb{P}^{X})$. By Assumption \ref{angle}, we have 
\[\Vert\Pi_{V_1}\Pi_{V_2}h\Vert\leq\rho_0\Vert\Pi_{V_2}h\Vert\ \text{ and }\ \Vert\Pi_{V_2}\Pi_{V_1}h\Vert\leq \rho_0\Vert\Pi_{V_1}h\Vert\]
for all $h\in L^2(\mathbb{P}^{X})$, which follows as in the proof of \eqref{eq:pa1}.  Applying this and \eqref{eq:smarg}, we obtain
\begin{align}
 &\big\Vert  \big(\Pi_{V_1}-\sum_{j=1}^k(\Pi_{V_1}\Pi_{V_2})^j(1-\Pi_{V_1})\big)(f_2-\Pi_{V_2}f_2)\big\Vert \nonumber\\
 &=\big\Vert\sum_{j=0}^k(\Pi_{V_1}\Pi_{V_2})^j\Pi_{V_1}(f_2-\Pi_{V_2}f_2)\big\Vert\nonumber\\
  &\leq\sum_{j=0}^k\rho_0^{2j}\Vert \Pi_{V_1}(f_2-\Pi_{V_2}f_2)\Vert\leq\frac{1}{1-\rho_0^2}\Vert h_1\Vert\psi_\Pi(V_2)\phi(V_2)\label{eq:backpa}.
\end{align} 
Since $\Pi_{V}\Pi_{V_2}f_2=\Pi_{V_2}f_2$ and $(\Pi_{V_2}f_2)_1=0$, we have $(\Pi_Vf_2)_1=
(\Pi_V(f_2-\Pi_{V_2}f_2))_1$. Applying this,  \eqref{eq:arpcon}, and \eqref{eq:backpa}, we conclude that
\begin{equation*}
\Vert(\Pi_Vf_2)_1\Vert\leq\frac{1}{1-\rho_0^2}\Vert h_1\Vert\psi_\Pi(V_2)\phi(V_2).
\end{equation*}
This completes the proof.
\end{proof}
\begin{proposition}\label{abrutemp} Let Assumption \ref{compass}, \ref{angle}, \ref{assinfty}, and \ref{HSce} be satisfied. Let $\phi(V_2)$, $\psi(V_2)$, and $h_1$ be as in Theorem \ref{thm3}. Moreover, suppose that $\Vert g_1\Vert_\infty\leq\varphi\sqrt{d_1}\Vert g_1\Vert$ for all $g_1\in V_1$. Then
\begin{multline}\label{eq:abrutempc}
 \mathbb{E}\left[ 1_{\mathcal{E}_\delta}\Vert(\hat{\Pi}_Vf_2)_1\Vert_n^2\right]
 \leq\frac{(1+\delta)^2}{(1-\delta)^3}\frac{2}{(1-\rho_{0}^2)^2}\bigg(\Vert h_1\Vert^2(\psi(V_2)\phi(V_2))^2\\
 +\left.\frac{1}{n}\Vert h_1\Vert^2\Vert(1-\Pi_{V_2})f_2\Vert^2_{\infty}\psi^2(V_2)+\phi^2(V_2)\frac{\varphi^2d_1}{n}\right).
\end{multline}
\end{proposition}
\begin{proof} The proof is similar to the proof of Proposition \ref{abrut}. Throughout the proof, suppose that the event $\mathcal{E}_\delta$ holds. Then $\hat{\Pi}_V$ is a well-defined map from $L^2(\mathbb{P}^X)$ to $V$. Let $\phi_1,\dots,\phi_{d_1}$ be an orthonormal basis of $V_1$. By repeating the arguments at the beginning of the proof of Lemma \ref{infty}, we obtain
\begin{equation*}
 \Vert\hat{\Pi}_{V_1}(f_2-\hat{\Pi}_{V_2} f_2)\Vert_n^2\leq \frac{1}{1-\delta}\sum_{j=1}^{d_1}\langle\phi_j,(1-\hat{\Pi}_{V_2})f_2\rangle_n^2.
\end{equation*}
Thus
\begin{align}
&\mathbb{E}\left[1_{\mathcal{E}_\delta} \Vert\hat{\Pi}_{V_1}(f_2-\hat{\Pi}_{V_2} f_2)\Vert_n^2\right]\nonumber \\
 &\leq\frac{2}{(1-\delta)}\mathbb{E}\left[\sum_{j=1}^{d_1}\langle\phi_{j,\Pi_2},(1-\hat{\Pi}_{V_2})f_2\rangle_n^2\right] \label{eq:abempc}\\
 &+\frac{2}{(1-\delta)}\mathbb{E}\left[\sum_{j=1}^{d_1}\langle\phi_j-\phi_{j,\Pi_2},(1-\hat{\Pi}_{V_2})f_2\rangle_n^2\right] ,\label{eq:abempze}
\end{align}
where 
\begin{equation*}
 \phi_{j,\Pi_2}(x_2)=\int \phi_j(x_1)\frac{p(x_1,x_2)}{p_1(x_1)p_2(x_2)}p_1(x_1)d\nu_1(x_1)
\end{equation*}
is the conditional expectation of $\phi_j(X_1)$ given $X_2=x_2$ (for $\mathbb{P}^{X_2}$-almost all $x_2$, by Assumption \ref{HSce}). The expectation in  \eqref{eq:abempc} is equal to
\begin{align*}
&\mathbb{E}\left[\sum_{j=1}^{d_1}\left( \int\left\langle\frac{p(x_1,\cdot)}{p_1(x_1)p_2(\cdot)},(1-\hat{\Pi}_{V_2})f_2\right\rangle_n\phi_j(x_1)p_1(x_1)d\nu_1(x_1) \right) ^2\right] \\
&\leq \int\mathbb{E}\left[\left\langle\frac{p(x_1,\cdot)}{p_1(x_1)p_2(\cdot)},(1-\hat{\Pi}_{V_2})f_2\right\rangle_n^2 \right]p_1(x_1)d\nu_1(x_1), 
\end{align*}
where we applied Bessel's inequality and Fubini's theorem in the last inequality.
Applying the fact that orthogonal projections are idempotent and self-adjoint and then the Cauchy-Schwarz inequality, this is
\begin{equation*}
\leq  \int\mathbb{E}\left[\left\Vert(1-\hat{\Pi}_{V_2})\frac{p(x_1,\cdot)}{p_1(x_1)p_2(\cdot)}\right\Vert_n^2\left\Vert(1-\hat{\Pi}_{V_2})f_2\right
\Vert_n^2\right]p_1(x_1)d\nu_1(x_1) .
\end{equation*}
Applying the projection theorem and then \eqref{eq:condjdens}, this is
\begin{align*}
&\leq  \int\mathbb{E}\left[\left\Vert(1-\Pi_{V_2})\frac{p(x_1,\cdot)}{p_1(x_1)p_2(\cdot)}\right\Vert_n^2\left\Vert(1-\Pi_{V_2})f_2\right
\Vert_n^2\right]p_1(x_1)d\nu_1(x_1)  \\
&\leq \frac{n-1}{n} \Vert h_1\Vert^2(\psi(V_2)\phi(V_2))^2+\frac{1}{n}\Vert h_1\Vert^2\Vert(1-\Pi_{V_2})f_2\Vert_{\infty}^2(\psi(V_2))^2.
\end{align*}
Now we turn to the expectation in \eqref{eq:abempze}. We have
\begin{align*}
\mathbb{E}\left[ \sum_{j=1}^{d_1}\langle\phi_j-\phi_{j,\Pi_2},(1-\hat{\Pi}_{V_2})f_2\rangle_n^2\right]&\leq \frac{\varphi^2d_1}{n}\mathbb{E}\left[ \Vert (1-\hat{\Pi}_{V_2})f_2\Vert^2_n\right]\\&\leq \frac{\varphi^2d_1}{n}\Vert (1-\Pi_{V_2})f_2\Vert^2.
\end{align*}
To prove the first inequality, first note that the $((1-\hat{\Pi}_{V_2})f_2)(X_2^i)$ depend only on $X_2^1,\dots,X_2^n$ and we have
\begin{equation*}
\mathbb{E}\left[ (\phi_j-\phi_{j,\Pi_2})(X^i)|X_2^1,\dots,X_2^n,X_1^{i'}\right]=\mathbb{E}\left[ (\phi_j-\phi_{j,\Pi_2})(X^i)|X_2^i \right]=0
\end{equation*}
for $i\neq i'$. This implies that the nondiagonal terms vanish. Next, apply the inequalities $\mathbb{E}[(\phi_j-\phi_{j,\Pi_2})^2(X)|X_2]\leq \mathbb{E}[\phi_j^2(X_1)|X_2]$ and \[\Big\Vert\sum_{j=1}^{d_1}\phi_j^2\Big\Vert_\infty\leq\varphi^2d_1,\] 
which follows from the bound $\Vert g_1\Vert_\infty\leq\varphi\sqrt{d_1}\Vert g_1\Vert$, for all $g_1\in V_1$, and \cite[Lemma 1]{BM}.
Thus we have shown that
\begin{multline}\label{eq:smargem}
 \mathbb{E}\left[ 1_{\mathcal{E}_\delta}\Vert\hat{\Pi}_{V_1}(f_2-\hat{\Pi}_{V_2} f_2)\Vert_n^2\right]
\leq\frac{2}{(1-\delta)}\left(\Vert h_1\Vert^2(\psi(V_2)\phi(V_2))^2\right.\\
\left.+\frac{1}{n}\Vert h_1\Vert^2\Vert(1-\Pi_{V_2})f_2\Vert_{\infty}^2(\psi(V_2))^2+\phi^2(V_2)\frac{\varphi^2d_1}{n}\right).
\end{multline}
The remaining arguments are as in the proof of Proposition \ref{abrut}. From \eqref{eq:backfcond} and Lemma \ref{Arvn} we have
\begin{equation}\label{eq:arpconem}
\big\Vert (\hat{\Pi}_Vh)_1-\big(\hat{\Pi}_{V_1}-\sum_{j=1}^k(\hat{\Pi}_{V_1}\hat{\Pi}_{V_2})^j(1-\hat{\Pi}_{V_1})\big)h\big\Vert_n \rightarrow 0
\end{equation}
as $k\rightarrow\infty$, for all $h\in L^2(\mathbb{P}^X)$. 
Applying \eqref{eq:pa1} and \eqref{eq:pa2} as in \eqref{eq:backpa}, we obtain
\begin{align}
 &\big\Vert  \big(\hat{\Pi}_{V_1}-\sum_{j=1}^k(\hat{\Pi}_{V_1}\hat{\Pi}_{V_2})^j(1-\hat{\Pi}_{V_1})\big)(f_2-\hat{\Pi}_{V_2}f_2)\big\Vert_n\nonumber\\
&\leq\frac{1}{1-\rho_{0,\delta}^2}\Vert \hat{\Pi}_{V_1}(f_2-\hat{\Pi}_{V_2}f_2)\Vert_n.\label{eq:backpaem}
\end{align}
From \eqref{eq:arpconem} and \eqref{eq:backpaem}, we conclude that
\begin{equation*}
\Vert(\hat{\Pi}_Vf_2)_1\Vert_n^2=\Vert(\hat{\Pi}_V(f_2-\hat{\Pi}_{V_2}f_2))_1\Vert_n^2
\leq\frac{1}{(1-\rho_{0,\delta}^2)^2}\Vert \hat{\Pi}_{V_1}(f_2-\hat{\Pi}_{V_2}f_2)\Vert_n^2.
\end{equation*}
Combining this with \eqref{eq:smargem} and \eqref{eq:prere3} gives \eqref{eq:abrutempc}. This completes the proof.
\end{proof}
\subsection{End of proof of Theorem \ref{thm3} and \ref{thm4}}
The only place where we modify the proof of Theorem \ref{thm2} is the analysis of the term
\begin{equation}\label{impr}
\frac{2}{(1-\delta)} \mathbb{E}\left[1_{\mathcal{E}_\delta}\Vert f_1-(\hat{\Pi}_Vf)_1\Vert_n^2 \right],
\end{equation}
which we bounded by
\begin{equation}\label{eq:upbf}
\frac{1+\delta}{(1-\delta)^2}\frac{6}{1-\rho_0^2}\left(\Vert f_1-\Pi_{V_1}f_1\Vert^2+\Vert f_2-\Pi_{V_2}f_2\Vert^2\right),
\end{equation}
by using Lemma \ref{lb2}. We show that, under the additional Assumptions \eqref{eq:condjdens} or \eqref{eq:condjdenspr}, one can replace the upper bound  \eqref{eq:upbf} by the ones given in Theorem \ref{thm3} and \ref{thm4}, respectively. To achieve this, we replace 
 Lemma \ref{lb2} by Proposition \ref{abrut} and \ref{abrutemp}.

In order to proof Theorem \ref{thm3} we decompose
\begin{equation*}
 f_1-(\hat{\Pi}_Vf)_1=f_1-(\Pi_Vf_1)_1+(\Pi_Vf_1)_1-(\hat{\Pi}_Vf_1)_1-(\hat{\Pi}_Vf_2)_1.
\end{equation*}
Thus
\begin{align*}
 \mathbb{E}\left[1_{\mathcal{E}_\delta}\|(f_1-(\hat{\Pi}_Vf_1)_1\|^2\right]
&\leq 3\mathbb{E}\left[\|f_1-(\Pi_Vf_1)_1\|_n^2\right]\\
&+3\mathbb{E}\left[1_{\mathcal{E}_\delta}\|(\Pi_Vf_1)_1-(\hat{\Pi}_Vf_1)_1\|_n^2\right]\\
&+3\mathbb{E}\left[1_{\mathcal{E}_\delta}\|(\hat{\Pi}_Vf_2)_1\|_n^2\right].
\end{align*}
The third term on the right-hand side is part of Proposition \ref{abrutemp}.
Using Lemma \ref{angleequiv} and the projection theorem, the first term can be bounded by
\begin{align}
 \mathbb{E}\left[\|f_1-(\Pi_Vf_1)_1\|_n^2\right]&=\|f_1-(\Pi_Vf_1)_1\|^2\nonumber\\
&\leq\frac{1}{(1-\rho_0^2)}\|f_1-\Pi_Vf_1\|^2\nonumber\\
&\leq\frac{1}{(1-\rho_0^2)}\|f_1-\Pi_{V_1}f_1\|^2.\label{eq:xvy}
\end{align}
Applying Proposition \ref{empangle} and the projection theorem, the second term can be bounded by
\begin{align*}
 &\mathbb{E}\left[1_{\mathcal{E}_\delta}\|(\Pi_Vf_1)_1-(\hat{\Pi}_Vf_1)_1\|_n^2\right]\\
&\leq\frac{(1+\delta)}{(1-\delta)}\frac{1}{(1-\rho_0^2)} \mathbb{E}\left[\|\Pi_Vf_1-\hat{\Pi}_Vf_1\|_n^2\right]\\
&\leq\frac{(1+\delta)}{(1-\delta)}\frac{1}{(1-\rho_0^2)} \mathbb{E}\left[\|f_1-\Pi_Vf_1\|_n^2\right]\\
&\leq\frac{(1+\delta)}{(1-\delta)}\frac{1}{(1-\rho_0^2)}\|f_1-\Pi_{V_1}f_1\|^2.
\end{align*}
This completes the proof of Theorem \ref{thm3}.
In order to proof Theorem \ref{thm4} we decompose
\begin{equation*}
 f_1-(\hat{\Pi}_Vf)_1=f_1-(\Pi_Vf_1)_1-(\Pi_Vf_2)_1+(\Pi_Vf)_1-(\hat{\Pi}_Vf)_1.
\end{equation*}
Thus
\begin{align*}
 \mathbb{E}\left[1_{\mathcal{E}_\delta}\|(f_1-(\hat{\Pi}_Vf_1)_1\|^2\right]
&\leq 3\mathbb{E}\left[\|f_1-(\Pi_Vf_1)_1\|_n^2\right]\\
&+3\mathbb{E}\left[\|(\Pi_Vf_2)_1\|_n^2\right]\\
&+3\mathbb{E}\left[1_{\mathcal{E}_\delta}\|(\Pi_Vf)_1-(\hat{\Pi}_Vf)_1\|_n^2\right].
\end{align*}
The first term on the right-hand side is bounded in \eqref{eq:xvy}, the second one in Proposition \ref{abrut}. Using the definition of $\mathcal{E}_\delta$ and Lemma \ref{angleequiv}, we obtain 
\begin{align*}
 \mathbb{E}\left[1_{\mathcal{E}_\delta}\|(\Pi_Vf)_1-(\hat{\Pi}_Vf)_1\|_n^2\right]
&\leq (1+\delta)\mathbb{E}\left[1_{\mathcal{E}_\delta}\|(\Pi_Vf)_1-(\hat{\Pi}_Vf)_1\|^2\right]\\
&\leq (1+\delta)\frac{1}{(1-\rho_0^2)} \mathbb{E}\left[1_{\mathcal{E}_\delta}\|\Pi_Vf-\hat{\Pi}_Vf\|^2\right].
\end{align*}
By Proposition \ref{infty}, this can be bounded by
\[\frac{(1+\delta)}{(1-\delta)^2}\frac{2}{(1-\rho_0^2)} \frac{\varphi^2d}{n}\left(\|f_1-\Pi_{V_1} f_1\|^2+\|f_2-\Pi_{V_2} f_2\|^2\right).\]
This completes the proof.\qed

\appendix 
\section{Proof of Lemma \ref{angleequiv}}\label{app2a}
We first show how (i) implies (ii) and (iii). Let $h_1\in \mathcal{H}_1$ and $h_2\in\mathcal{H}_2$. Then by (i) we have $\Vert h_1+h_2\Vert^2\geq \Vert h_1\Vert^2-2\varrho\Vert h_1\Vert\Vert h_2\Vert+\Vert h_2\Vert^2$ and (ii) follows from the inequality $2\Vert h_1\Vert\Vert h_2\Vert\leq \Vert h_1\Vert^2+\Vert h_2\Vert^2$, while (iii) follows from $2\varrho\Vert h_1\Vert\Vert h_2\Vert\leq \varrho^2\Vert h_1\Vert^2+\Vert h_2\Vert^2$.

Next, we show how (ii) implies (i). Let $0\neq h_1\in \mathcal{H}_1$ and $0\neq h_2\in\mathcal{H}_2$. We may assume without loss of generality that $\Vert h_1\Vert=\Vert h_2\Vert=1$ and that $\langle h_1,h_2\rangle\geq 0$. Then by (ii) we have $2-2\langle h_1,h_2\rangle=\|h_1-h_2\|^2\geq 2(1-\varrho)$ which gives (i). 

Finally, suppose that (iii) is true. Let $0\neq h_1\in \mathcal{H}_1$ and $0\neq h_2\in\mathcal{H}_2$. Again, we may assume that $\Vert h_1\Vert=\Vert h_2\Vert=1$. Then by (iii) we have $1-\langle h_1,h_2\rangle^2=\|h_1-\langle h_1,h_2\rangle h_2\|^2\geq 1-\varrho^2$ which gives (i). This completes the proof.\qed

\section{A feasible estimator}\label{app2f}
In this appendix, we show that estimators based on the condition $(1/n)\sum_{i=1}^ng_1(X_1^i)=0$ have (up to a constant and a term of smaller order) the same risk bound as our estimators based on the condition $\mathbb{E}\left[ g_1(X_1)\right]=0$. We only sketch the main arguments in the case $W_1=V_1$.
Suppose that we choose $U_1\subset L^2(\mathbb{P}^{X_1})$ and $V_2\subset L^2(\mathbb{P}^{X_2})$, where $V_2$ contains all constant functions. Let 
$V_1'=\lbrace g_1\in U_1|(1/n)\sum_{i=1}^n g_1(X_1^i)=0 \rbrace$ and $V_1=\lbrace g_1\in U_1|\mathbb{E}\left[ g_1(X_1)\right]=0 \rbrace$. Since $V_2$ contains all constants, we have $V=V_1+V_2=V_1'+V_2$. This implies that the first components of $\hat{f}_{V_1+V_2}$ and $\hat{f}_{V_1'+V_2}$ in $V_1$ and $V_1'$, respectively, differ only by the constant
\[\frac{1}{n}\sum_{i=1}^n(\hat{f}_{V})_1(X_1^i).\]
The risk of this constant can be bounded as follows. By the bound $(x+y)^2\leq 2x^2+2y^2$, we have
\begin{align*}
 &\mathbb{E}\left[ 1_{\mathcal{E}_\delta}\left(\frac{1}{n} \sum_{i=1}^n(\hat{f}_{V})_1(X_1^i)\right) ^2\right] \\
 &\leq 2\mathbb{E}\left[1_{\mathcal{E}_\delta} \left(\frac{1}{n}\sum_{i=1}^n(\hat{f}_{V})_1(X_1^i)-f_1(X_1^i)\right) ^2\right]\\
 &+2\mathbb{E}\left[1_{\mathcal{E}_\delta} \left(\frac{1}{n}\sum_{i=1}^nf_1(X_1^i)\right) ^2\right].
\end{align*}
Applying the Cauchy-Schwarz inequality and the fact that the $f_1(X_1^i)$ are independent and centered, this can be bounded by
\[2\mathbb{E}\left[ 1_{\mathcal{E}_\delta}\Vert (\hat{f}_{V})_1-f_1\Vert_n^2\right]+ \frac{2\Vert f_1\Vert^2}{n}.\]
Now apply Lemma \ref{lb2} to the first term.

\section{Proof of Lemma \ref{varphis} and \ref{varphi}}\label{app212}
First, we prove Lemma \ref{varphis}. In Section \ref{a1dim}, we have shown that
\[\left\|g_1\right\|_\infty^2\leq \varphi_1^2d_1\left\|g_1\right\|^2\]
and
\[\left\|g_{j2}\right\|_\infty^2\leq \varphi_2^2d_{j2}\left\|g_{j2}\right\|^2\]
for all $g_1\in V_1$, $g_{j2}\in V_{j2}$, $1\leq j\leq q-1$, with $\varphi_1^2=2/c$ and $\varphi_2^2=2/c$.
Now let $g=g_1+g_2\in V$. Suppose that $g_2=\sum_{j=1}^{q-1}g_{2j}$ is the decomposition satisfying \eqref{eq:kobcond}. Applying the above bounds, the Cauchy-Schwarz inequality, and Assumption \ref{ac}, we obtain
\[
\left\| g_2\right\|_\infty\leq \sum_{j=1}^{q-1}\varphi_2\sqrt{d_{2j}}\left\|g_{2j}\right\|
\leq \frac{\varphi_2}{\sqrt{1-\epsilon_2}}\sqrt{\sum_{j=1}^ {q-1} d_{2j}}\left\|g_2\right\|.
\]
Applying again the Cauchy-Schwarz inequality and then Assumption \ref{angle} and Lemma \ref{angleequiv}, we conclude that
\begin{align}
 \Vert g_1+g_2\Vert_\infty &\leq\varphi_1\sqrt{d_1}\Vert g_1\Vert+\frac{\varphi_2}{\sqrt{1-\epsilon_2}}\sqrt{\sum_{j=1}^ {q-1} d_{2j}}\Vert g_2\|\nonumber \\
 & \leq\sqrt{\frac{\varphi_1\vee\varphi_2}{(1-\epsilon_2)(1-\rho_0)}}\sqrt{d_1+\sum_{j=1}^ {q-1} d_{2j}}\Vert g_1+g_2\Vert.\label{eq:stma}
\end{align}
This completes the proof of Lemma \ref{varphis}. The proof of Lemma \ref{varphi} is similar.
In \cite{BM2}, it is shown that 
\[\left\|g_1\right\|_\infty^2\leq (r_1+1)^2m_1\int_0^1g_1^2(x_1)dx_1\]
and 
\[\left\|g_{j2}\right\|_\infty^2\leq (r_2+1)^2m_2\int_0^1g_{j2}^2(x_{j2})dx_{j2}\]
for all $g_1\in V_1$, $g_{j2}\in V_{j2}$, $1\leq j\leq q-1$. This implies that
\[\left\|g_1\right\|_\infty^2\leq \varphi_1^2d_1\left\|g_1\right\|^2\]
and
\[\left\|g_{j2}\right\|_\infty^2\leq \varphi_2^2d_{j2}\left\|g_{j2}\right\|^2\]
with $\varphi_1^2=2(r_1+1)/c$ and $\varphi_2^2=2(r_2+1)/c$. 
Now proceed as above. 
This completes the proof.
\qed
\section{Proof of Corollary \ref{corlemmab}}\label{app2b}
\begin{lemma}\label{lemmab} Let Assumption \ref{HSce} and \ref{ac} be satisfied. 
Suppose that \eqref{eq:jointdens} and \eqref{eq:jointdens2} are satisfied.
Then \eqref{eq:condjdenspr} is satisfied with \[\psi_\Pi(V_2)=\sqrt{C_3}\sqrt{\sum_{k=1}^{q-1}d_{2k}^{-2\beta}}\] and
\begin{equation*}
 h_1(x_1)=\sqrt{\sum_{k=1}^qh^2_{1k}(x_1)}+
 \sqrt{\frac{c'}{1-\epsilon_2}\int\bigg(\frac{p_X(x_1,x_2)}{p_{X_1}(x_1)p_{X_2}(x_2)}\bigg)^2p_{X_2}(x_2) dx_2},
\end{equation*}
where $c'=\sum_{j,k=1,j\neq k}^{q-1}\Vert h'_{jk}\Vert^2$. Note that $h_1\in L^2(\mathbb{P}^{X_1})$, by Assumption \ref{HSce}.
\end{lemma}
\begin{proof}
Let $x_1$ be fixed (such that $p_X(x_1,\cdot)/(p_{X_1}(x_1)p_{X_2}(\cdot))\in L^2(\mathbb{P}^{X_2})$, which is satisfied for $\mathbb{P}^{X_1}$-almost all $x_1$, by Assumption \ref{HSce}). By the projection theorem, the expression
\begin{equation*}
\int\left(\frac{p_X(x_1,x_2)}{p_{X_1}(x_1)p_{X_2}(x_2)}-g(x_2)\right)^2p_{2}(x_2)dx_2,
\end{equation*}
subject to the constraints $g\in H_2$, is minimized by $r$. Suppose that $r=\sum_{k=1}^{q-1}r_{k}$ is the decomposition such that \eqref{eq:kobcond} is satisfied (note that we omit the dependence of $r$ and the $r_{k}$ on $x_1$). For $k=1,\dots,q-1$, we have 
\begin{equation}\label{eq:condexf}
 \mathbb{E}\left[\frac{p_X(x_1,X_2)}{p_{X_1}(x_1)p_{X_2}(X_2)}\bigg|X_{2k}=x_{2k}\right] =\frac{p_{X_1,X_{2k}}(x_1,x_{2k})}{p_{X_1}(x_1)p_{X_{2k}}(x_{2k})}.
\end{equation}
Thus the $r_{k}$ satisfy the $q-1$ equations
\begin{align*}
 &r_{k}(x_{2k})=\\&\frac{p_{X_1,X_{2k}}(x_1,x_{2k})}{p_{X_1}(x_1)p_{X_{2k}}(x_{2k})}-\sum_{j=1,j\neq k}^{q-1}\int r_{j}(x_{2j})\frac{p_{X_{2j},X_{2k}}(x_{2j},x_{2k})}{p_{X_{2j}}(x_{2j})p_{X_{2k}}(x_{2k})}p_{X_{2j}}(x_{2j})dx_{2j},
\end{align*}
for $\mathbb{P}^{X_{2k}}$-almost all $x_{2k}$, $1\leq k\leq q-1$ (note again that we omit the dependence of the $r_{k}$ on $x_1$). By \eqref{eq:jointdens}, \eqref{eq:jointdens2}, and the Cauchy-Schwarz inequality, the first and the second term on the right hand side are contained in $\mathcal{H}(\beta,h_{1k}(x_1))$ and $\mathcal{H}(\beta,\sum_{j=1,j\neq k}^{q-1}\Vert r_{j}\Vert_{L^2(\mathbb{P}^{X_{2j}})}\Vert h'_{jk}\Vert)$, respectively. We conclude that
\begin{align*}
 &\Vert r-\Pi_{V_2}r\Vert_{L^2(\mathbb{P}^{X_2})}\\
&\leq \sum_{k=1}^{q-1}\Vert r_{k}-\Pi_{V_{2k}}r_{k}\Vert_{L^2(\mathbb{P}^{X_{2k}})}\\
 &\leq \sum_{k=1}^{q-1}C_3\left( h_{1k}(x_1)+\sum_{j=1,j\neq k}^{q-1}\Vert r_{j}\Vert_{L^2(\mathbb{P}^{X_{2j}})}\Vert h'_{jk}\Vert\right) d_{2k}^{-\beta}.
\end{align*}
Applying the Cauchy-Schwarz inequality and  Assumption \ref{ac}, this is bounded by
\begin{equation*}
\leq \sum_{k=1}^{q-1}C_3d_{2k}^{-\beta} \left(h_{1k}(x_1)+\sqrt{\frac{\Vert r\Vert^2_{L^2(\mathbb{P}^{X_2})}}{1-\epsilon_2}\sum_{j=1,j\neq k}^{q-1}\Vert h'_{jk}\Vert^2}\right).
\end{equation*}
Applying the Cauchy-Schwarz inequality again and the fact that orthogonal projections lower the norm, we obtain the claimed $\psi_\Pi(V_2)$ and $h_1(x_1)$. This completes the proof.
\end{proof}

\section{Proof of Lemma \ref{trlem}}\label{app2h}
We only proof (iii), since (i) and (ii) are standard. By the spectral theorem, there exists an orthogonal matrix $V$ and nonnegative real numbers $\lambda_1(B),\dots,\lambda_{k_1}(B)$ such that
\begin{equation}\label{eq:spthm}
 B=V^T\operatorname{diag}(\lambda_1(B),\dots,\lambda_{k_1}(B))V.
\end{equation} 
Now, by the Cauchy-Schwarz inequality, each entry of a matrix is bounded by the operator norm of that matrix. In particular, we have $|(VAV^T)_{jk}|\leq \Vert VAV^T\Vert_{\operatorname{op}}=\Vert A\Vert_{\operatorname{op}}$ for all $j,k$, since $V$ is orthogonal. Applying \eqref{eq:spthm}, part (i) of this Lemma, the fact that the $\lambda_j(B)$ are nonnegative, and the previous argument, we obtain
\begin{align*}
 |\operatorname{tr}(AB)|&=\Big|\sum_{j=1}^{k_1}(VAV^T)_{jj}\lambda_j(B)\Big|\\&\leq \max_{j=1,\dots,k_1}|(VAV^T)_{jj}|\operatorname{tr}(B)\leq \Vert A\Vert_{\operatorname{op}}\operatorname{tr}(B).
\end{align*}
This completes the proof.\qed

\section{Proof of \eqref{eq:pr1}}\label{app2ha} In this appendix, we prove \eqref{eq:pr1}. As mentioned in the proof of Theorem \ref{thm1}, the main arguments are taken from \cite[page 139 and 140]{B}. We define the event $\mathcal{A}=\lbrace\Vert \hat{f}_1\Vert_\infty\leq k_n \rbrace$. Then
\begin{align*}
&\mathbb{E}\left[ \Vert f_1-\hat{f}_{1}^*\Vert^2\right]=\mathbb{E}\left[(1_{\mathcal{E}_\delta}1_{\mathcal{A}}+
1_{\mathcal{E}_\delta}1_{\mathcal{A}^c}+1_{\mathcal{E}_\delta^c}) \Vert f_1-\hat{f}_{1}^*\Vert^2\right]\\
&\leq \mathbb{E}\left[1_{\mathcal{E}_\delta} \Vert f_1-\hat{f}_1\Vert^2\right]
+ \mathbb{E}\left[1_{\mathcal{E}_\delta}1_{\mathcal{A}^c} \Vert f_1\Vert^2\right]
+ \mathbb{E}\left[1_{\mathcal{E}_\delta^c}(\Vert f_1\Vert+k_n)^2\right].
\end{align*}
Thus it remains to consider the last two expressions. By Theorem \ref{mfg}, the last one is bounded by
\[2^{3/4}(\Vert f_1\Vert+k_n)^2 d\exp\left( -\kappa\frac{n\delta^2}{\varphi^2d}\right).\]
Consider the other one. By Assumption \ref{assinfty}, we have $\Vert \hat{f}_1\Vert^2_\infty\leq \varphi^2d\Vert \hat{f}_1\Vert^2$. If $\mathcal{E}_{\delta}$ holds, then
\begin{align*}
\Vert \hat{f}_1\Vert^2_\infty\leq\frac{\varphi^2d}{(1-\delta)}\Vert \hat{\Pi}_{W_1}(\hat{\Pi}_V\mathbf{Y})_1\Vert^2_n
\leq \frac{\varphi^2d}{(1-\delta)}\Vert (\hat{\Pi}_V\mathbf{Y})_1\Vert^2_n,
\end{align*}
where we applied the definition of $\mathcal{E}_{\delta}$ and the fact that projections lower the norm.
By Proposition \ref{empangle}, the last expression is bounded by 
\[\frac{(1+\delta)\varphi^2d}{(1-\delta)^2(1-\rho_0^2)}\Vert \hat{\Pi}_V\mathbf{Y}-g_2\Vert^2_n,\]
for $g_2\in V_2$ arbitrary.
Using $\Vert \hat{\Pi}_V\mathbf{Y}-g_2\Vert_n\leq \Vert \hat{\Pi}_V (f-g_2)\Vert_n+\Vert \hat{\Pi}_V\boldsymbol{\epsilon}\Vert_n\leq \Vert f-g_2\Vert_n+\Vert\boldsymbol{\epsilon}\Vert_n$ and Markov's inequality, we conclude that
\begin{align*}
 \mathbb{P}(\mathcal{E}_\delta\cap\mathcal{A}^c)&\leq 
 \mathbb{P} \left( \frac{(1+\delta)\varphi^2d}{(1-\delta)^2(1-\rho_0^2)}(\Vert f-g_2\Vert_n+\Vert\boldsymbol{\epsilon}\Vert_n)^2>k_n^2\right) \\
 &\leq \frac{2(1+\delta)\varphi^2d(\Vert f-g_2\Vert^2+\sigma^2)}{(1-\delta)^2(1-\rho_0^2)k_n^2}.
\end{align*}
 Letting $g_2=\Pi_{V_2}f$, this completes the proof.\qed

\section*{Acknowledgements}
Finally, I sincerely would like to thank Prof. Enno Mammen for his support and guidance during the preparation of this paper.

\bibliographystyle{plain}
\bibliography{lit}
\end{document}